\documentclass[hidelinks,onefignum,onetabnum]{siamart251104}

\usepackage{lipsum}
\usepackage{amsfonts}
\usepackage{graphicx}
\usepackage{epstopdf}
\usepackage{algorithm,algorithmic}
\usepackage{verbatim}
\usepackage{multirow}
\usepackage{subcaption}
\usepackage{enumitem}
\usepackage{amsmath,booktabs,ctable,threeparttable}
\usepackage{amssymb,boxedminipage}
\usepackage{microtype}
\allowdisplaybreaks[4]
\ifpdf
  \DeclareGraphicsExtensions{.eps,.pdf,.png,.jpg}
\else
  \DeclareGraphicsExtensions{.eps}
\fi

\newsiamremark{remark}{Remark}
\newsiamremark{hypothesis}{Hypothesis}
\crefname{hypothesis}{Hypothesis}{Hypotheses}
\newsiamthm{claim}{Claim}
\newsiamremark{fact}{Fact}
\crefname{fact}{Fact}{Facts}

\crefname{hypothesis}{Hypothesis}{Hypotheses}
\usepackage{amsopn}

\DeclareMathOperator{\Range}{\mathcal{R}}

\newcommand{\Complex}{\mathbb{C}}

\newcommand{\norm}[2][]{\Vert {#2} \Vert_{#1}}

\newcommand{\abs}[1]{\vert {#1} \vert}

\headers{A REFINED CJ--SS--RR METHOD}{Z. Jia and T. Liu}

\title{A refined CJ--SS--RR method with a reliable removal approach of
spurious Ritz values for the Hermitian eigenvalue problem\thanks{Submitted to the editors DATE.
\funding{The work was supported by the National Natural Science Foundation of China (NSFC) under
grant 12571404.}}}

\author{Zhongxiao Jia\thanks{Corresponding author. Department of Mathematical Sciences, Tsinghua University, 100084 Beijing, China
  (\email{jiazx@tsinghua.edu.cn}).}
\and Tianhang Liu\thanks{Department of Mathematical Sciences, Tsinghua University, 100084 Beijing, China
  (\email{lth21@mails.tsinghua.edu.cn}).}
}

\begin{document}
\maketitle

\begin{abstract}
  Under the hypothesis that the deviations of the desired eigenvectors of the matrix 
  $A$ from
  the underlying subspace tend to zero, the Ritz vectors may not converge 
  and have poor or little accuracy.  This phenomenon is not unusual and particularly occurs 
  when the associated Ritz values are close, which is independent of 
  the eigenvalue distribution of $A$.
  %There are more Ritz values that approximate
  %the same desired eigenvalues, or some of the Ritz values do not approximate any eigenvalue(s).
  For the (block) SS--RR methods, there are possibly
  {\em more} Ritz values that converge to the same desired eigenvalue(s) counting
  multiplicity in the region of interest, meaning that some of the Ritz values must
  be spurious and the corresponding
  residual norms of the Ritz pairs may not be small.
  %The second spurious case is that some of the Ritz values
  %in the region do not approximate {\em any} eigenvalues.
  Consequently, the (block) SS--RR methods
  including the CJ--SS--RR method cannot base on the corresponding residual norms to
  effectively identify if the Ritz values
  in the region are genuine or spurious.
  %These severely affect the effectiveness and efficiency of the (block) SS--RR methods.
  This paper proposes refined SS--RR, abbreviated as SS--RRR, methods based on the refined Rayleigh--Ritz
  projection that compute the eigenpairs of large matrices with the eigenvalues located
  in the given region. We present a new approach to accurately implement the RRR methods more efficiently
  than ever before for a general subspace.
  Exploiting the unconditional convergence of the refined Ritz vectors when
  the subspace is sufficiently
  accurate, we propose a tune-free removal approach to effectively remove
  spurious Ritz values with a rigorous theory supported, and develop a restarted CJ--SS--RRR algorithm.
  Numerical experiments show that the restarted CJ--SS--RRR algorithm is more efficient
  and effective than the restarted CJ--SS--RR algorithm.
\end{abstract}



\section{Introduction}
Consider the following eigenvalue problem of a large-scale Hermitian matrix $A \in \mathbb{C}^{n \times n}$:
\begin{displaymath}
  A x = \lambda x \mbox{\ \  with \ } x^H x = 1,
\end{displaymath}
where the superscript $H$ denotes the conjugate transpose of a vector or matrix.
Given an interval $[a, b]$ contained in the spectrum interval of $A$, we are required
to compute all the eigenvalues $\lambda\in [a, b]$ counting multiplicities and their
corresponding eigenvectors.
Such kind of problems arises from numerous applications, such as
computing eigenpairs within an energy window near the Fermi level \cite{Martin_Electronic_Structure},
extracting electromagnetic modes in target frequency bands \cite{Jin_FEM_EM},
and locating unstable/weakly damped modes in prescribed spectral regions \cite{Trefethen_Embree},
to mention a few.

For matrices of small or medium size, there have been some well-established methods
for computing all the eigenvalues and eigenvectors, such as the QR algorithm and divide-and-conquer algorithm
\cite{bai2000,Golub_Matrix,Stewart_Eigen}, which are considered to be ``direct'' methods.
For large-scale matrices,
a specific portion of eigenvalues and eigenvectors is desired in applications.
In the Hermitian case, there have been a few standard projection methods available, which include
the subspace iteration method for computing several largest eigenvalues in magnitude and the
corresponding eigenvectors or invariant subspace,
the Lanczos method and the LOBPCG method for computing extreme eigenpairs,
and the rational Krylov method and the Jacobi--Davidson method for computing the eigenvalues closest to a given
target point $\tau$ and the corresponding eigenvectors
\cite{bai2000,Golub_Matrix,Parlett_Symmetric,Saad_Eigen,Stewart_Eigen}.

Over the past two decades, a new class of numerical methods based on contour integration and rational filtering have been developed to compute all the eigenpairs of a large matrix (pair) with the
eigenvalues in a given region or interval.
In 2003, Sakurai and Sugiura \cite{SS_Hankel} proposed the SS--Hankel method to solve the generalized eigenvalue problem of non-Hermitian matrix pairs based on contour integration and numerical quadrature.
The method computes a series of the zeroth to higher-order moment matrices via numerical quadrature to approximate the contour integrals, generates a certain Hankel matrix pair, and solves the resulting
generalized eigenvalue problem to obtain approximate eigenpairs of the original matrix (pair).
The SS--Hankel method is equivalent to performing projection with a non-orthonormal basis of the search subspace \cite{SS_review}, and is numerical unstable due to the ill-conditioning of Hankel matrices \cite{SS_RR}.
Subsequently, the SS--RR method \cite{SS_RR} and the SS--Arnoldi method \cite{SS_Arnoldi} were
proposed to improve the numerical stability and increase the accuracy of approximate eigenpairs.
The SS--RR method generates an orthonormal basis of the search subspace via
a truncated singular value decomposition (TSVD), and performs projection onto the resulting
subspace with {\em possibly smaller} dimension.
The SS--Arnoldi method \cite{SS_Arnoldi} generates an orthonormal
basis of the search subspace using a block Arnoldi-type process, realizes the Rayleigh--Ritz projection,
and extract approximate eigenpairs from the search subspace.
All these methods include single vector and block versions \cite{block_SS_RR,SS_theory2010,SS_Arnoldi}.
A software has been available \cite{SS_package} for these SS-type methods.

The FEAST method was initially proposed by Polizzi \cite{FEAST2009} to solve generalized Hermitian eigenvalue problems using contour integration and numerical quadrature, and has been rigorously
analyzed and developed by Tang and Polizzi \cite{FEAST_SI}.
Unlike the (block) SS--RR method, the FEAST method only constructs the zeroth-order moment matrix
by numerical quadrature, which is an approximate spectral projector corresponding
to the eigenvalues in the interval of interest; the method then
applies subspace iteration to the approximate spectral projector,
generates a sequence of increasingly better subspaces with fixed dimension,
and obtains approximate eigenpairs from
them until convergence.
Subsequently, G\"{u}ttel et al.\cite{FEAST_Zolotarev} proposed a variant of the FEAST method that employs
Zolotarev rational functions to construct possibly more accurate approximate spectral projectors. In the meantime,
the FEAST method has been extended to non-Hermitian matrices using the right and left approximate spectral projectors \cite{FEAST_non_Hermitian} or the oblique projection \cite{FEAST_oblique}.
A contour integral-based FEAST software has been available \cite{FEAST_package}.
%In recent years, instead of using numerical quadrature or rational filtering, Jia and Zhang %\cite{CJ_FEAST_cross,CJ_FEAST_augmented,CJ_FEAST_GSVD} have exploited Chebyshev--Jackson (CJ) series
%expansions to construct approximate projectors of the eigenvalues corresponding to the singular values
%and generalized singular values in a given interval, and proposed CJ-FEAST type solvers for the SVD and %GSVD problems, which are much more numerically
%stable and efficient than corresponding contour integral-based FEAST solvers.

At each iteration the (block) SS--RR and FEAST eigensolvers require solving a number of large
complex shifted linear systems with the quadrature nodes as shifts.
Since the region of interest is inside the spectrum of $A$, these linear systems may be
highly indefinite, thereby posing severe challenges for commonly used Krylov
iterative solvers such as the GMRES method
and the BiCGStab method \cite{Saad_Linear}.
Moreover, once a shift is close to some eigenvalue of $A$, the corresponding
shifted linear systems are
ill conditioned, causing that the accuracy of computed solutions
may not suffice to ensure the convergence of the contour integral-based eigensolvers
or the methods cannot achieve a reasonably user-prescribed convergence tolerance
\cite{CJ_FEAST_cross,CJ_FEAST_augmented,CJ_FEAST_GSVD}, even though they are solved accurately
in finite precision arithmetic.

To address the afore-mentioned issues, several variants have emerged recently. For the SVD and generalized
SVD (GSVD)
problems with the singular values and generalized singular values in a given interval,
Jia and Zhang \cite{CJ_FEAST_cross,CJ_FEAST_augmented,CJ_FEAST_GSVD} have
proposed the two CJ-FEAST SVDsolvers and one GSVDsolver that exploit the
Chebyshev--Jackson (CJ) polynomial series expansion to approximate certain
underlying spectral projectors. The solvers avoid
solving any shifted linear system, only compute matrix-matrix
products, and are thus more stable and accurate.
They have established the convergence theory of the CJ-FEAST SVDsolvers and GSVDsolver, and developed
the algorithms with several issues addressed,
which are numerically confirmed to be much more efficient and robust than the corresponding
contour integral-based FEAST solvers for the SVD and GSVD
computations. These CJ-FEAST methods are directly adaptable to the eigenvalue problems of
Hermitian matrices and generalized Hermitian positive definite matrix pairs.
Subsequently, the authors of this paper have used CJ series expansions to approximate
higher-order moments involved in the (block) SS--RR methods, established compact quantitative approximation
results, and proposed the CJ--SS--RR algorithm by considering several
important issues untouched or incorrectly taken for granted previously. For instance,
instead of monomial moments, we have resorted
to the Chebyshev moments to greatly improve numerical stability and robustness of the SS--RR method \cite{CJ_SS_RR}.
The CJ--SS--RR method approximates the contour integrals using polynomials instead of
rational functions, thereby avoiding
solving difficult large linear systems and improving the attainable
accuracy of the SS--RR methods greatly. The solvers have numerically been
confirmed to be much more
efficient than the corresponding contour integral-based solvers.

The common framework of projection methods consists of two main
ingredients: the construction of a search subspace and the choice of a projection method
or extraction approach used to extract approximate eigenpairs from the subspace.
The FEAST and SS--RR methods for the eigenvalue problem
are standard Rayleigh--Ritz projection methods \cite{Parlett_Symmetric,Stewart_Eigen}
and compute Ritz approximations, and they differ in the way that constructs a subspace.
The convergence of the FEAST-type methods for the eigenvalue problem has been established in
\cite{FEAST_Zolotarev,FEAST_non_Hermitian,FEAST_SI}. For the SVD and GSVD problems,
Jia and Zhang \cite{CJ_FEAST_cross,CJ_FEAST_augmented,CJ_FEAST_GSVD} have proved that the Ritz pairs
obtained by the CJ-FEAST solvers must converge to the desired SVD and GSVD components
when the the degree of CJ series expansion is reasonably large.
The results hold for the eigenvalue problems of Hermitian matrices and generalized Hermitian positive definite matrix pairs.

Nevertheless, the convergence results on the Ritz pairs obtained by the FEAST-type methods
do {\em not} hold for the (block) SS--RR methods.
The FEAST-type methods are special instances of the SS--RR methods, and employ only the
zeroth moment. There are remarkable differences between them in convergence and implementations.
For the FEAST-type methods, under the natural hypothesis that all the desired
eigenvectors are not deficient in the starting subspace,
the {\em whole} search subspace converges to an invariant
subspace of the underlying matrix (pair), which ensures the
unconditional convergence of the Ritz pairs when a reasonable numerical
quadrature is used.
However, it is {\em not} the case for the SS--RR methods that employ the zeroth to higher-order
moments. It is known that their underlying search subspaces approximate block Krylov subspaces as
numerical quadrature converges \cite{SS_theory2010}. Particularly, Imakura et al.\cite{SS_theory2016}
have proved that such kind of search subspace is indeed a block Krylov subspace when the integral curve is
a circle and numerical quadrature is the trapezoidal rule.

Remarkably, it is extraordinarily {\em rare} that the whole block Krylov subspace
converges to an invariant
subspace of the matrix (pair). This is a fundamental distinction between the subspaces used
by the FEAST-type methods and those used by the SS--RR methods. Consequently, the Ritz vectors
obtained by the SS--RR methods may converge irregularly and
even may fail to converge under the hypothesis that the deviations of
desired eigenvectors from the underlying subspace tend to zero \cite{refined_RR_theory,jiastewart2001,Stewart_Eigen}.
The cause is that (i)
there may be more Ritz values, which are multiple or very close,
that approximate the same eigenvalue counting multiplicity and (ii)
there may be Ritz values that do {\em not} approximate any eigenvalues; see \cite{refined_RR_theory,jiastewart2001,Stewart_Eigen} for details. In case (i), the Ritz values
include genuine and spurious ones, and the Ritz values in case (ii) are simply spurious.
In case (i),
the spurious Ritz values severely affect the accuracy of all the corresponding Ritz vectors
in an irregular manner, and the residual norms of these Ritz pairs may or may not be small
even though the Ritz values
themselves have already converged. In case (ii), the residual norm of each spurious Ritz
pair is definitely not small. As a result,
when measuring
the accuracy of an approximate eigenpair in terms of its residual norm,
the Rayleigh--Ritz projection method itself cannot identify genuine and spurious Ritz values,
an intrinsic deficiency of the method.
%Such phenomena affect not only
%the reliability of the Rayleigh--Ritz projection but also its convergence and efficiency.

In computations, we have indeed observed that there are some already converged
Ritz values in the region of interest but not all the residual norms of Ritz pairs computed by
the SS--RR methods are small \cite{SS_2023}, meaning that the corresponding Ritz vectors do
not converge and some of the Ritz values may be spurious.
%the methods themselves cannot judge if the Ritz values have converged and thus
%cannot remove spurious Ritz values.
A nowadays commonly used approach to distinguishing the genuine Ritz values
from the spurious ones in the region is to
simply compute the residual norms and remove the ones whose corresponding
residual norms are not small. Based on the elaboration in the above paragraph,
such approach is unreliable and may not work well because the residual norms of those
genuine Ritz pairs may or may not be small whenever there are spurious Ritz values.
This indicates that
it does not exist a theoretical threshold for judging residual norms so as to effectively
identify
genuine and spurious Ritz values, and we may either remove some genuine Ritz
values or retain some spurious Ritz values for a user-prescribed threshold.
Consequently, counting the number of Ritz values in the region of interest using such approach
may deliver incorrect results when the methods terminate.
Such an improper identification may
%There are some other approaches to treating spurious Ritz values  %\cite{block_SS_RR,SS_Arnoldi,SS_2023,SS_RR,FEAST_oblique}, and we will describe them in \cref{sec:3} and
%show why they do not work well and
severely affect the effectiveness and efficiency of
the restarted SS--RR algorithms. %are mathematically incorrect and numerically infeasible.
%Moreover, the residual norms of the Ritz vectors corresponding to the spurious Ritz values are typically
%large.
%Consequently, when using the residual norms to judge convergence, these spurious Ritz pairs will impair the
%convergence behavior of the method.

%The researchers of \cite{SS_2023,FEAST_oblique} have proposed several approaches to remove the spurious
%Ritz pairs.
%However, as is elaborated above,
The possible occurrence of spurious Ritz pairs is intrinsic in the Rayleigh--Ritz
projection for a general subspace, and the residual norm-based removal approach
is empirical and has no theoretical support, and
its performance relies entirely on an unsound user-prescribed tolerance.
Therefore, it is appealing to consider other projection methods to fix the
aforementioned deficiencies of
the Rayleigh--Ritz projection and compute the desired eigenvalues
in the region more reliably and efficiently. As it will turn out,
the refined Rayleigh--Ritz projection method, also
called the refined extraction approach,
which was initially proposed in \cite{jia1997refined,jia1998refined},
developed in \cite{jia1999refined,jia2000refined,refined_RR_theory} and popularized in
\cite{bai2000,jiastewart2001,Stewart_Eigen}, can resolve these issues perfectly
and converge faster. Applying
the refined Rayleigh--Ritz projection to the subspace generated by the SS-RR methods including
CJ--SS--RR method, we
obtain refined SS-RR (SS--RRR) methods, which includes the refined CJ--SS--RR (CJ--SS--RRR) method.
We shall propose a new approach to
compute several refined Ritz vectors accurately, which is numerically stable but can be
substantially more efficient than the approach
in \cite{jia2000refined} for a general subspace. Using the refined Ritz pairs obtained by the SS--RRR methods,
we propose a reliable residual norm-based removal approach that distinguishes
the genuine refined Ritz values from the spurious ones once they start to converge.
Unlike the SS--RR methods,
we establish a rigorous theoretical basis for the new removal approach.
Our removal approach based on CJ--SS--RRR is tune free, requires no auxiliary user-prescribed
threshold, cluster those approximating nearby refined Ritz values automatically,
and remove the spurious ones reliably if any
as they start to converge.

The paper is organized as follows. In \cref{sec:2}, we review the CJ--SS--RR method.
\Cref{sec:3} considers the removal issue of spurious Ritz values and reveals intrinsic limitations of
the (block) SS--RR methods.
In \cref{sec:4,sec: implementation}, we present the CJ--SS--RRR method, consider its
implementations, and
show how to effectively remove those extra refined Ritz vectors and retain
genuine Ritz values exactly.
Numerical experiments are reported in \cref{sec:6} to illustrate the effectiveness
of the proposed removal approach and the higher efficiency of the restarted
CJ-SS-RRR algorithm than the restarted CJ-SS-RR algorithm. Finally, \cref{sec:7} concludes the paper.

Throughout the paper, denote by $\norm{\cdot}$ the vector 2-norm or the matrix spectral norm,
by $\kappa(\cdot)$ the spectral norm condition number of a matrix,
by $\mathcal{R}(X)$ the column space of a matrix $X$,
by $\lambda(A)$ the set of the eigenvalues of $A$,
and by $\lambda_{\min}$ and $\lambda_{\max}$ the minimum and maximum eigenvalues of $A$, respectively.

\section{The CJ--SS--RR method}\label{sec:2}

% In this section, we review the theoretical results of CJ--SS--RR method established in \cite{CJ_SS_RR} and some results about CJ series approximation retrieved from \cite{CJ_FEAST_cross}.
We review the CJ series expansion and the CJ--SS--RR method in \cite{CJ_SS_RR}.
Suppose that $[\lambda_{\min},\lambda_{\max}]=[-1,1]$.
For a given interval $[a, b] \subset [-1, 1]$, define the step function
\begin{equation} \label{eq: def: step-function}
  h(t) = \begin{cases}
    1,           & t \in (a, b) , \\
    \frac{1}{2}, & t \in \{ a, b \}, \\
    0,           & t \in [-1, 1] \setminus [a, b].
  \end{cases}
\end{equation}
Suppose that $n_{ev}$ is the number of the eigenvalues $\lambda\in [a, b]$ of
$A$ counting multiplicities,
and choose positive integers $M, \ell$ such that $M\ell\geq n_{ev}$.

For a real algebraic polynomial $p$ of degree not exceeding $M-1$,
let the operator $F_d$ be the $d$-degree CJ polynomial series
expansion approximation to $p(t)h(t)$ \cite{CJ_SS_RR}:
\begin{equation} \label{eq: def: F_d}
  F_d(p)(t) := \frac{c_0}{2} +  \sum_{j=1}^d \rho_{j, d} c_j T_j(t) \approx p(t)h(t),
\end{equation}
where  $T_j$ is the $j$-degree Chebyshev polynomial of the first kind \cite{ApproxIntroduction}, the
coefficients
\begin{equation} \label{eq: def: Chebyshev coefficient}
  c_j = \frac{2}{\pi} \int_{-1}^1 \frac{p(t)h(t) T_j(t)}{\sqrt{1-t^2}}  {\rm d} t = \frac{2}{\pi} \int_a^b \frac{p(t) T_j(t)}{\sqrt{1-t^2}} {\rm d} t,
\end{equation}
and the Jackson damping factors (cf. \cite{eig_count_Saad,Jackson_damp})
\begin{equation} \label{eq: def: Jackson damping factor}
  \rho_{j, d} = \frac{\sin((j+1)\alpha_d)}{(d+2)\sin(\alpha_d)} + \left( 1-\frac{j+1}{d+2} \right) \cos(j\alpha_d), \quad \text{where } \alpha_d = \frac{\pi}{d+2}.
\end{equation}

For a given starting matrix $V\in \Complex^{n\times \ell}$ of full column rank, define the approximations
to the zeroth to $(M-1)$-order moment matrices by
% \begin{equation} \label{eq: def: polynomial approximation of S}
%   S = (S_0, S_1, \ldots, S_{M-1}), \quad \text{where } S_k = F_d(p_k)(A)V, \ p_k(t) = T_k \left( \frac{2t-a-b}{b-a} \right),
% \end{equation}
% where $V\in \Complex^{n\times \ell}$ is a starting matrix.
\begin{equation} \label{eq: def: polynomial approximation of S}
  \begin{aligned}
    &S_k = F_d(p_k)(A)V, \quad k = 0, 1, \ldots, M-1, \\
    &S = (S_0, S_1, \dots, S_{M-1}),
  \end{aligned}
\end{equation}
where $p_k(t) = T_k \left( (2t-a-b)/(b-a) \right)$ are shifted-and-scaled $k$-degree
Chebyshev polynomials.
The column space $\Range(S)$ is the search subspace used in the CJ--SS--RR method.
We remark that, in the SS--RR methods,
$S_i, i=0, 1, \ldots, M-1$ are generated by some contour integral-based numerical
quadrature, e.g., the the trapezoidal rule, where %shifted-and-scaled
monomials are used in the $M$ contour integrals \cite{block_SS_RR,SS_RR}.
To make the CJ--SS--RR method work correctly, the degree $d$ must be properly large.
Based on \cite[Theorem 4.6 and section 5.2]{CJ_SS_RR},
the authors in \cite{CJ_SS_RR} give a
practical estimation of CJ series degree:
\begin{equation}\label{eq: degree estimate for practice}
  d\geq \left\lceil \frac{D\pi^2}{(b-a)^{4/3}} + \frac{\pi^2(M-1)^2}{K^2(b-a)} \right\rceil - 2
\end{equation}
with $D \in [1, 8]$ and $K \in [1, 10]$ depending on the eigenvalue distribution of $A$.

Generally, $[\lambda_{\min},\lambda_{\max}]\not=[-1,1]$.
Note that the linear transformation
\begin{displaymath}
  l(t) = \frac{2t - \lambda_{\max} - \lambda_{\min}}{\lambda_{\max} - \lambda_{\min}}
\end{displaymath}
maps $[\lambda_{\min}, \lambda_{\max}]$ to $[-1, 1]$. In applications, it is
only necessary to replace $t, a, b$ and $A$
by $l(t),l(a), l(b)$ and $l(A)$, respectively.

The abbreviation ``RR'' in the (block) SS--RR method stands for
the Rayleigh--Ritz projection method: %, which is described as follows:
\begin{equation} \label{eq: def: Rayleigh--Ritz projection}
  \begin{cases}
    \tilde{x} \in \mathcal{U}, \\
    A \tilde{x} - \tilde{\lambda} \tilde{x} \perp \mathcal{U}, \\
  \end{cases}
\end{equation}
where $\tilde{\lambda}$ and $\tilde{x}$ with $\|\tilde x\|=1$
are called the Ritz values and Ritz vectors of $A$ with respect to
a given search subspace $\mathcal{U}$.
Let the columns of $U$ form an orthonormal basis of $\mathcal{U}$.
Then the Ritz pairs are computed by solving the projected eigenvalue problem:
\begin{equation} \label{eq: def: Rayleigh--Ritz procedure}
  U^H A U y = \tilde{\lambda} y, \quad \tilde{x} = U y \mbox{\ \ with \ \ } \|y\|=1.
\end{equation}
When $\mathcal{U}=\Range(S)$ with $S$ generated by a contour integral-based numerical
quadrature, the resulting method is a SS--RR method; we obtain the CJ--SS--RR
method when $S$ is defined by \eqref{eq: def: polynomial approximation of S}.

\cref{alg: CJSSRR} describes a restarted CJ--SS--RR algorithm \cite{CJ_SS_RR}.
In step \ref{step: CJSSRR: convergence test}, the relative residual norm of a Ritz pair $(\tilde{\lambda}, \tilde{x})$ is defined as
\begin{equation}\label{relres}
  \frac{\norm{A\tilde{x} - \tilde{\lambda}\tilde{x}}}{\norm{A}},
\end{equation}
and the stopping criteria in \cite{FEAST_oblique} are used to test convergence;
a Ritz pair with $\tilde{\lambda}\in [a, b]$,
and relative residual norm smaller than some user-prescribed
threshold $\delta$ is identified as desired. In computations, the denominator
$\|A\|$ can be replaced by $(\|A\|_1\|A\|_{\infty})^{1/2}$, where $\|\cdot\|_1$
and $\|\cdot\|_{\infty}$ denote the 1-norm and the infinity norm, respectively.
The algorithm is stopped when the relative residual
norms of the desired Ritz pairs are smaller than
a prescribed tolerance $tol$.
In step \ref{step: CJSSRR: restart}, our restart strategy is the same as
the approach proposed in
\cite{SS_theory2016,SS_review,SS_restart1} and takes
$V^{(k+1)} = F_d(1)(A)V^{(k)}$, which is the first $\ell$ columns of $S^{(k)}$,
given that $p_0(t) = 1$ in \cref{eq: def: polynomial approximation
of S}.

In the original SS--RR methods, the moment matrices $S_0,\ldots,S_{M-1}$ similar
to those in \eqref{eq: def: polynomial approximation of S} are approximations to
the moment matrices generated by ill-conditioned monomials.
In finite precision arithmetic, for numerical stability,
the corresponding step 4
uses the TSVD to generate an effective subspace;
see \cite{SS_package,SS_2023} and the references therein. In \cite{CJ_SS_RR},
we have replaced the monomials by the much better conditioned Chebyshev polynomials
in \eqref{eq: def: polynomial approximation of S}.

\begin{algorithm}
\caption{The restarted CJ--SS--RR algorithm}
\label{alg: CJSSRR}
\begin{algorithmic}[1]
  \REQUIRE The Hermitian matrix $A$, the real interval $[a, b]$, the CJ series degree $d$, the moment number $M$, an $n$-by-$\ell$ starting
  matrix $V=V^{(0)}$ in \eqref{eq: def: polynomial approximation of S}
  with $M\ell \geq n_{ev}$, and the stopping tolerance $tol$.
  \STATE Compute the CJ series coefficients by \cref{eq: def: Chebyshev coefficient,eq: def: Jackson damping factor}.
  \FOR {$k = 0, 1, \ldots$}
    \STATE Compute $S^{(k)}=(S_0^{(k)},S_1^{(k)},\ldots,S_k^{(k)})$
    by \cref{eq: def: polynomial approximation of S}.
    \STATE Compute the QR factorization: $S^{(k)} = U^{(k)}R^{(k)}$. \label{step: CJSSRR: QR}
    \STATE Apply the Rayleigh--Ritz projection \cref{eq: def: Rayleigh--Ritz procedure} to  $\mathcal{R}(U^{(k)})$, and compute the Ritz pairs $(\tilde{\lambda}_i^{(k)}, \tilde{x}_i^{(k)}), i = 1, 2, \ldots, M\ell$. \label{step: CJSSRR: RR}
    \STATE Compute the relative residual norms \eqref{relres}, and test convergence. \label{step: CJSSRR: convergence test}
    \STATE If not converged, set the updated starting matrix $V^{(k+1)}$ to
    be the first $\ell$ columns of $S^{(k)}$, and go to step 2. \label{step: CJSSRR: restart}
  \ENDFOR
  \RETURN The converged $n_{ev}$ Ritz pairs $(\tilde{\lambda}_i^{(k)}, \tilde{x}_i^{(k)})$.
\end{algorithmic}
\end{algorithm}

\section{The removal approaches of spurious Ritz values in the SS--RR methods}
\label{sec:3}

Before we start the analysis,
we should point out that the occurrence of spurious Ritz values is universal,
regardless of whether or not $A$ is Hermitian.
Therefore, in this and next sections, we consider a general matrix.

%The (block) SS--RR and CJ--SS--RR methods use the Rayleigh--Ritz projection \cref{eq: def: Rayleigh--Ritz
%projection} to obtain the Ritz pairs $(\tilde{\lambda}_i, \tilde{x}_i)$.
It is well known that the necessary condition for the convergence of approximate eigenvectors
by a projection method
is that all the desired eigenvectors
are sufficiently close to the search subspace $\mathcal{U}$ \cite{jiastewart2001,Stewart_Eigen}.
We call such a $\mathcal{U}$ sufficiently accurate.
When the subspace dimension $M> n_{ev}$ and $\mathcal{U}$ is not sufficiently accurate,
it is well possible that there are either
more or fewer than $n_{ev}$ Ritz values inside the region of interest;
for $\mathcal{U}$ sufficiently accurate, there are at least $n_{ev}$ Ritz values
in the region because there are $n_{ev}$ Ritz values that converge unconditionally
\cite{jiastewart2001,Stewart_Eigen}.

%In this case, we should select at most $n_{ev}$ Ritz pairs from the $M\ell$ ones.
%In this case,
%a straightforward approach is to select the Ritz pairs with Ritz values in the interval of interest.
%However, the number of these Ritz pairs may be still larger than $n_{ev}$.

As we have addressed, there may be more Ritz values that
approximate the same eigenvalue of $A$ counting multiplicity,
meaning that these approximating Ritz values include both genuine and spurious when they
are in the region.
Some intuitive examples can be found in \cite{refined_RR_theory},
where there are multiple Ritz values
that are equal to a simple eigenvalue when the desired eigenvector exactly lies in the subspace or,
more generally, there are more nearly multiple Ritz values that approximate the same simple
eigenvalue of $A$ when the desired eigenvector almost lies in the subspace.
In these two cases, the Ritz vectors
are either not uniquely determined or, though unique, may have little accuracy, causing
that the residual norms of these Ritz pairs
may not be small though the Ritz values are already an exact eigenvalue or very accurate approximations
to the eigenvalue. Obviously, part of the approximating Ritz values
is spurious. The second kind of spurious Ritz value is that it is in the region but
does {\em not}
approximate any eigenvalue of $A$, which appears more likely when the region
is very inside the whole spectrum region of $A$ since the Ritz values are equal to
Rayleigh quotients $u^H A u$ and could be more likely
impersonated even though $u$'s do not
approximate any eigenvector of $A$.

At first glance,
it seems straightforward to identify the second
kind of spurious Ritz values by checking the residual norm of the corresponding Ritz pair,
which must not be small since it does not approximate
any eigenvalue of $A$. Unfortunately, it is not the case because the residual norms
of Ritz values of the first kind may not be small either, as we have elaborated above.
%But this issue is actually hard to settle
%down reliably in computation because one generally does not know the type of Ritz values in advance.

%
%The extra Ritz pairs, whose Ritz values in $[a, b]$ with corresponding Ritz vectors that
%do not necessarily approximate the eigenvectors, are called spurious in \cite{SS_2023,FEAST_oblique} but %without any precise definition.
In what follows we are devoted to a
reliable and accurate removal approach of spurious Ritz values in order
to retain only genuine ones.
We give the following definition for $A$ Hermitian, which is
straightforwardly adaptable to a general $A$ by replacing
the interval by a given region in the complex plane.

\begin{definition}
  \label{def: spurious Ritz vector}
  Given a Hermitian matrix $A$, an interval $[a, b] \subset [\lambda_{\min}, \lambda_{\max}]$ and a subspace $\mathcal{U}$, let $(\tilde{\lambda}, \tilde{x})$ be a Ritz pair of $A$ with respect to $\mathcal{U}$.
  Then
  \begin{itemize}
    \item[{\rm (i)}] $\tilde{x}$ is called a potential Ritz vector if $\tilde{\lambda} \in [a, b]$;
   % \item An eigenvector $x$ is desired if its corresponding eigenvalue $\lambda$ is in $[a, b]$;
    \item[{\rm (ii)}] For $\tilde{\lambda} \in [a, b]$, the Ritz vector $\tilde{x}$ is called desired if
    the residual norm of $(\tilde\lambda, \tilde x)$ is below some threshold $\delta$; % and the angle between $\tilde{x}$ and $x$ is the smallest among all potential Ritz vectors;
    \item[{\rm (iii)}] $\tilde{x}$ is called a spurious Ritz vector if it is a potential but not desired one.
  \end{itemize}
\end{definition}

The removal approach in \cite{SS_package,CJ_SS_RR,FEAST_oblique}
detects the spurious Ritz vectors by checking the relative residual norms of the potential Ritz vectors:
the one with a relative residual norm larger than a user-prescribed threshold $\delta$
is identified as spurious.

Unfortunately, as we have stated above, for the Rayleigh--Ritz method,
whenever there are spurious Ritz values in the region or interval of interest,
all the Ritz vectors corresponding to both the genuine and spurious Ritz values
may have poor or little accuracy when $\mathcal{U}$ contains
sufficiently accurate approximations to the desired eigenvectors, which is even true
when the desired eigenvectors $x\in \mathcal{U}$ exactly. In other words,
for the first kind of spurious Ritz values, there may be no Ritz vector corresponding to the (nearly) multiple Ritz values %, which include
%the genuine and spurious ones,
that is guaranteed to converge even if the desired eigenvector is sufficiently close to or exactly
lies in the subspace. These mean that there may be no desired Ritz vector for a given small $\delta$
even the Ritz values have converged and one,  therefore, cannot identify whether
the (nearly) multiple Ritz values are genuine or spurious. The difficulty is there does not exist
any reliable threshold $\delta$. As a consequence, for {\em any} given small $\delta$, one may remove
some genuine Ritz values or retain some spurious Ritz values, failing to find
the $n_{ev}$ desired eigenpairs correctly. For the second kind of spurious Ritz values, the
residual norms of Ritz pairs are definitely not small, but we do not know it is the Ritz vector
or both of the Ritz value and Ritz vector that lead to not small residual norms. Consequently,
the Rayleigh--Ritz method cannot base on residual norms to
include or exclude those Ritz values in the region.
Therefore, as a whole,
there is no reliable residual norm-based approach to identifying the spurious Ritz values and to
making the Rayleigh--Ritz method converge regularly,
and any residual norm-based removal approach of spurious Ritz
values and vectors has no solid theoretical guarantee.

A reliable identification of approximate eigenvalues
% or (generalized) singular values
in the region or interval is vital to the effectiveness and efficiency of
the (block) SS--RR methods including the CJ--SS--RR method.
Whenever either fewer or more than $n_{ev}$ Ritz pairs are
retained, the methods either fail to find all the $n_{ev}$ eigenpairs or delays
the convergence.

We now pay special attention to
step \ref{step: CJSSRR: QR} of a restarted SS--RR algorithm similar to \cref{alg: CJSSRR}.
For numerical stability, the original algorithm uses the TSVD of $S$ to retain the effective
information of $\mathcal{R}(S)$ in finite precision arithmetic; see
\cite{SS_package,SS_2023} and the references therein. There are a few deficiencies in such treatment.

\begin{itemize}
\item[(i)] %The exact determination of $n_{ev}$ is a very hard issue,
%and there has been no reliable and efficient approach to carrying out this task up to now.
Everything is possible when employing the TSVD to determine the numerical rank of $S$
for a user-prescribed truncation tolerance without considering
the mathematical property of $S$:
the numerical rank of $S$ may be bigger, smaller than or equal to $n_{ev}$, and it may well not equal the mathematical rank of $S$.
If the numerical rank is smaller than $n_{ev}$, the (block)
SS--RR methods cannot find the $n_{ev}$ desired eigenpairs.
%In terms of \cref{def: spurious Ritz vector},
%The approach simply retains all the $\tilde\lambda\in [a,b]$ without checking their
%residual norms.

\item[(ii)] The (block) SS--RR methods generate
the moments $S_i, i=0, 1, \ldots, M-1$ in \eqref{eq: def: polynomial approximation of S} using %shifted-and-scaled
monomials, and the resulting $S$ is typically very ill conditioned for a generic $\ell$.
As a result, as has been particularly addressed in \cite{CJ_SS_RR}, when ${\rm dim}(\Range(S))=M\ell$,
i.e., the mathematical rank of $S$ is $M\ell$,
the numerical rank obtained by the TSVD may
well be smaller than $n_{ev}$. This means that the methods artificially loses effective information
severely, as has been numerically confirmed for a number of test problems
in \cite{CJ_SS_RR}.

\item[(iii)] Extremely important is that,
as a basic principle, since the approximate eigenpairs obtained by a projection method
are {\em mathematically independent} of choices of basis vectors of $\mathcal{R}(S)$, one
must fully exploit $\mathcal{R}(S)$ rather than artificially reduce it
by using a truncation because of ill-conditioned basis vectors. Therefore,
it is vital to construct a well-conditioned basis of $\mathcal{R}(S)$ and
fully exploit all the effective information on $\mathcal{R}(S)$
in finite precision arithmetic.
\end{itemize}

Being directed against the above (iii), the authors in \cite{CJ_SS_RR} have proposed
to use much better-conditioned shifted-and-scaled Chebyshev moments to
replace ill-conditioned monomial
moments, as indicated by \eqref{eq: def: polynomial approximation of S}.
The numerical experiments have shown that,
in finite precision arithmetic,
the matrix $S$ in \eqref{eq: def: polynomial approximation of S} is quite well conditioned and always
has full column rank, meaning that the full effective information
on $\mathcal{R}(S)$ is retained.
%Notice that $S$
%in \eqref{eq: def: polynomial approximation of S} generates the same
%subspace $\mathcal{R}(S)$ as that using the shifted-and-scaled monomial moments.
More precisely, using the TSVD of $S$ retains all the effective
information and obtains a subspace whose dimension is equal to $M\ell$ for a threshold as small as
the level of machine precision.
%However, it also implies that we do not distinguish between genuine and
%spurious Ritz values since we simply take all the Ritz values in the region.
%On the contrary, for $S$ generated by the %shifted-and-scaled
%monomial moments, the numerical rank of $S$ is often reduced greatly and may be smaller than $n_{ev}$
%for a threshold at the level of machine precision.

Besides the residual norm-based removal approach, there is a simple TSVD approach \cite{SS_package}.
As we have mentioned previously, in computations
the SS--RR methods exploit the TSVD of $S$ to obtain an effective subspace and project
the eigenvalue problem of the original matrix (pair) onto the subspace. The TSVD approach simply retains all
the Ritz values in the region without checking their residual norms.
Rigorously speaking, it does not remove any Ritz values in the region and maintains
the natural form of a SS--RR method.

In summary, the residual norm-based removal approach of spurious
Ritz values are unreliable and may identify
the spurious Ritz pairs incorrectly for a given threshold $\delta$.
Its performance is unliable and has no theoretical guarantee.
Even if an appropriate threshold $\delta$ is set by chance, as soon as the matrix $A$ or the interval $[a, b]$
changes, the threshold may no longer be appropriate, thus leading to the incorrect detection of the
desired and spurious Ritz vectors.

%In summary, the aforementioned two removal approaches have severe deficiencies
%for the (block) SS--RR methods, and may work unreliably.

\section{The (block) SS--RRR methods}\label{sec:4}
\label{subsec: refined Rayleigh--Ritz projection}

%In order to deal with the troubles caused by spurious Ritz values,
We now review the refined Rayleigh--Ritz projection method proposed by Jia \cite{jia1997refined,jia1998refined,jia2000refined,refined_RR_theory} and
systematically accounted for in \cite{bai2000,Stewart_Eigen,book:VanDerVorst2002}:
For a Ritz value $\tilde{\lambda}$ approximating the eigenvalue $\lambda$,
find a unit length vector $\hat{x} \in \mathcal{U}$ that satisfies the residual optimality
\begin{equation} \label{eq: def: refined Rayleigh--Ritz projection}
  \norm{(A - \tilde{\lambda} I) \hat{x}}
  = \min_{w \in \mathcal{U} , \norm{w} = 1} \norm{(A - \tilde{\lambda} I) w},
\end{equation}
and use it to approximate $x$. The vector
$\hat x$ is called a refined Ritz vector or, more generally, a refined
eigenvector approximation to $x$ when $\tilde\lambda$ is replaced by any available
approximation to $\lambda$.  Specifically, when $\mathcal{U}=\mathcal{R}(S)$ generated by
the SS--RR methods or by the CJ--SS--RR method, we obtain the refined SS--RR, abbreviated as SS--RRR,
methods and the CJ--SS--RRR method.

Jia \cite{refined_RR_theory} has proved that $\norm{(A - \tilde{\lambda} I) \hat{x}} < \norm{(A - \tilde{\lambda} I) \tilde{x}}$, provided  $\norm{(A - \tilde{\lambda} I) \tilde{x}}\not=0$. Moreover, if there is other Ritz value close to $\tilde{\lambda}$,
then $\norm{(A - \tilde{\lambda} I) \hat{x}} \ll \norm{(A - \tilde{\lambda} I) \tilde{x}}$,
meaning that $\hat{x}$ is a much more accurate approximation to the eigenvector $x$ than $\tilde{x}$;
it is proved in \cite{jiastewart2001,Stewart_Eigen} that $\tilde\lambda\rightarrow \lambda$,
$\hat{x}\rightarrow x$, and the relative residual norm $\norm{(A - \tilde{\lambda} I) \hat{x}}/\|A\|$
tends to zero and is the same order size
as the deviation of $x$ from $\mathcal{U}$;
see \cite[Theorem 7.1]{jiastewart2001}. This is completely different
from the Ritz vector $\tilde{x}$, whose uniqueness requires that $\tilde{\lambda}$ be simple;
if there are multiple spurious $\tilde{\lambda}$, there are more than one Ritz vectors $\tilde{x}$'s as
approximations of $x$ with little accuracy, leading to the failure of the method. More generally,
if there are more nearby Ritz values
that approximate the same $\lambda$,
then the corresponding Ritz vectors either have little accuracy or wander with
uncertain accuracy,
so that the residual norms of Ritz pairs either are not small or uncertain.

In contrast, the convergence and accuracy
of the refined Ritz vector has nothing to do with spurious Ritz value(s), while those of
Ritz vectors critically depends on if there is a spurious Ritz value. Therefore,
in terms of \cref{def: spurious Ritz vector},
unlike the Ritz vectors, a potential refined Ritz vector $\hat x$
is always a desired one, and no spurious refined Ritz vector.
The residual norms of all the desired refined Ritz vectors tend to zero, and
the refined Ritz vectors converge to the desired eigenvectors $x$, as long as the deviations of
the desired eigenvectors $x$ from $\mathcal{U}$ tend to zero.

For a cluster of the nearby Ritz values consisting of possibly
genuine and spurious ones that approximate the same eigenvalue $\lambda$,
their corresponding refined Ritz vectors all converge
to the eigenspace associated with $\lambda$.
However, a new question arises:
If there are spurious Ritz values in the cluster,
the number of the corresponding
refined Ritz vectors must be bigger than the dimension of the eigenspace.
This means that these refined Ritz vectors must be nearly linearly dependent and
some of them are extra and must be removed.
Precisely,
suppose that $A$ is diagonalizable and $\lambda$ is $n_1$ multiple and the number of nearby Ritz values in
the cluster is $\hat n_1>n_1$. Then the numerical rank of the matrix consisting of the $\hat n_1$ refined Ritz
vectors for some {\em proper} threshold must equal $n_1$ under the natural
hypothesis that $\mathcal{U}$ contains
sufficiently good approximations to the $n_1$-dimensional eigenspace associated with
$\lambda$, and thus there are surely $\hat n_1-n_1$ spurious Ritz values. We thus remove $\hat n_1-n_1$ extra
refined Ritz vectors, aretain only $n_1$ ones, and obtain $n_1$ desired eigenpairs.
%These theoretical results and analysis provide us a principle to remove
%spurious Ritz values and extra refined Ritz vectors, as we show quantitatively
%in the next section.

In the next section, we will detail and quantify the above results,
and propose a reliable approach to removing possible
spurious Ritz values and extra refined Ritz vectors without a user-prescribed threshold.
The removal approach is tune free and only checks the residual norms of refined Ritz vectors.
It is unlike the SS--RR methods, which need an unsound and sensitive threshold $\delta$ that
is not directly related with the stopping tolerance $tol$.

\section{Effective and efficient implementations of the SS--RRR and CJ--SS--RRR methods}
\label{sec: implementation}

In \cref{subsec: computation of refined Ritz vectors,subsec: removal of extra refined vectors}, unless otherwise specified,
we assume that $A$ is a general diagonalizable matrix rather than require it to be Hermitian.
We will first present a new approach to compute the refined Ritz vectors
more efficiently than the SVD approach to be described
and more accurately than the cross-product matrix-based approach
in \cite{jia2000refined} for a general subspace,
% which is more efficient than that proposed in \cite{jia2000refined},
and then propose an approach to remove the spurious Ritz values
by removing the extra refined Ritz vectors.
% which is equivalent to
% removing the extra refined Ritz vectors,
% using the refined Rayleigh--Ritz projection \eqref{eq: def: refined Rayleigh--Ritz projection} for a general diagonalizable matrix $A$.

\subsection{A new more efficient computational approach of the refined Ritz vectors}
\label{subsec: computation of refined Ritz vectors}

We consider efficient and accurate computation of the refined Ritz vectors.
Suppose that $\mathcal{U}$ is an $M\ell$-dimensional subspace with $M\ell \geq n_{ev}$,
% the columns of $U$ form an orthonormal basis of $\mathcal{U}$,
the columns of $U$ form an orthonormal basis of $\mathcal{U}$,
and $(\tilde{\lambda}, \tilde{x})$ is a Ritz pair with respect to $\mathcal{U}$.
Then the minimization problem \eqref{eq: def: refined Rayleigh--Ritz projection} amounts to
\begin{equation} \label{eq: refined Rayleigh--Ritz projection in matrix form}
\norm{(A - \tilde{\lambda} I)\hat{x}}=  \norm{(A - \tilde{\lambda} I) U \hat{z}}
  = \min_{z \in \Complex^{M\ell}, \norm{z} = 1} \norm{(A - \tilde{\lambda} I) U z},
\end{equation}
where the refined Ritz vector $\hat{x} = U \hat{z}$ with $\hat z$ being the right
singular vector of $(A - \tilde{\lambda} I) U$ corresponding to its smallest
singular value \cite{jia2000refined,refined_RR_theory}.

For those $U$ generated by the (block) Arnoldi
process, one can compute $\hat{z}$ very efficiently and accurately via the SVD of
a small $(M+1)\ell\times M\ell$ (block) Hessenberg matrix \cite{jia1997refined,jia1998refined}.
For a general $U$, the most accurate and direct
way is to compute the compact SVD of the overly rectangular matrix $(A - \tilde{\lambda} I) U$
and obtain $\hat{z}$.
Suppose that there are $p$ Ritz values in the given region with
$p \leq M\ell$.
% It is worthwhile to point out that, in the beginning stage,
% it is well possible that $0\leq p<n_{ev}$ but we must have $p\geq n_{ev}$ as the $n_{ev}$
% desired eigenvectors tend to $\Range(U)$ because there are at least $n_{ev}$ Ritz values among
% $M\ell$ that converge
% the desired eigenvalues unconditionally \cite{jiastewart2001,Stewart_Eigen}.
% \textcolor{red}{If $p<n_{ev}$, the $p$ refined Ritz vectors are not enough for the refined RR method for
% the concerning eigenvalue problem, which is unlike the past problem
% that a known number of eigenpairs are desired. How to effectively and efficiently
% treat such an important issue is unclear.
% Certainly, one can first compute Ritz vectors to fix this deficiency until
% $p\geq n_{ev}$.}
Then, for $A$ and $\tilde{\lambda}$ real, the compact SVD computations of $p$ matrices
$(A - \tilde{\lambda} I) U$ cost
\begin{equation}\label{svdflops}
p(2n(M\ell)^2 + 11(M\ell)^3) = 2np(M\ell)^2 + 11p(M\ell)^3
\end{equation}
flops \cite[p.493]{Golub_Matrix}. We call this the SVD approach.
A more efficient but possibly little bit less accurate alternative is
to compute the eigenvector $\hat{z}$ of
the cross-product matrix $U^H(A - \tilde{\lambda} I)^H(A - \tilde{\lambda} I) U$ corresponding to its
smallest eigenvalue \cite{jia2000refined}.
%Alternatively, one more efficient but slightly less accurate
%way is to the eigendecomposition of the cross-product
%matrix $U^H(A - \tilde{\lambda} I)^H(A - \tilde{\lambda} I) U$ and obtain $\hat{z}$, which
%is the eigenvector of this matrix corresponding to its smallest eigenvalue \cite{jia2000refined}.

We now propose a new approach to compute $\hat{z}$ for a general $U$,
which is more efficient but as accurate as the SVD approach when computing the $p$ refined Ritz vectors.
Write \begin{displaymath}
  (A - \tilde{\lambda} I) U
  = \begin{pmatrix}  U & AU \end{pmatrix}
  \begin{pmatrix} -\tilde{\lambda}I \\ I \end{pmatrix},
\end{displaymath}
and compute the compact QR factorization
\begin{equation}\label{qrf}
  \begin{pmatrix}  U & AU \end{pmatrix}
  = Q \begin{pmatrix} R_{11} & R_{12} \\ 0 & R_{22} \end{pmatrix}
  = Q \begin{pmatrix}  J & H \end{pmatrix},
\end{equation}
where $R_{11}, R_{12}, R_{22}$ are $M\ell$-by-$M\ell$ matrices and $J, H$ are $2M\ell$-by-$M\ell$ matrices.
Then \eqref{eq: refined Rayleigh--Ritz projection in matrix form} is equivalent to
\begin{equation} \label{eq: refined Rayleigh--Ritz projection with QR}
  \norm{(A - \tilde{\lambda} I) U \hat{z}}
  = \min_{z \in \Complex^{M\ell}, \norm{z} = 1} \norm{(H - \tilde{\lambda} J) z}.
\end{equation}
Therefore, $\hat{z}$ is the right singular vector of the reduced matrix $H - \tilde{\lambda} J$
corresponding to its smallest singular value.
We only compute the Householder QR factorization \eqref{qrf} in factored form
once without explicitly forming
$Q$ at the cost of $2n(2M\ell)^2-2(2M\ell)^3/3\approx 2n(2M\ell)^2$ flops for $A$ real and
 $M\ell\ll n$
\cite[p.249]{Golub_Matrix}. The computation of the QR factorization \eqref{qrf}
and the compact SVDs of the $p$ reduced small
matrices $(H - \tilde{\lambda} J)$'s totally costs no more than
$$
2n(2M\ell)^2 + p(2(2M\ell)(M\ell)^2 + 11(M\ell)^3) = 8n(M\ell)^2 + 15p(M\ell)^3
$$
flops, much cheaper than the SVD approach for $M\ell \ll n$ and $p \gg 4$; see \eqref{svdflops}.
%For $A$ or $\tilde\lambda$ complex, the conclusion is similar.

\subsection{The removal approach of spurious Ritz values and extra refined Ritz vectors}
\label{subsec: removal of extra refined vectors}

As has been already known, if there is a spurious Ritz value in the cluster of nearly
Ritz values,
then the refined Ritz vectors must be nearly linearly dependent,
while the corresponding Ritz vectors may be messy since they
may converge irregularly or may not converge even if the deviation of the desired
eigenspace from $\mathcal{U}$ tends to zero.
Therefore, unlike the Ritz vectors,
the unconditional convergence of refined Ritz vectors enables us
to effectively check the closeness of the refined Ritz vectors,
identify spurious Ritz values and remove extra refined Ritz vectors.
As for the spurious Ritz values in the region that do not approximate any
eigenvalue of $A$, their associated refined Ritz vectors must not converge
to any eigenvector and the corresponding residual norms must not be small.
Therefore, it is straightforward to reliably remove such kind of spurious Ritz values
by checking the sizes of residual norms, which is unlike in the SS--RR methods.

In what follows it suffices to consider the case that
the refined Ritz vectors associated with a cluster of nearby Ritz values
approximating the same eigenvalue $\lambda$, and show how
to remove all the possible spurious Ritz values and extra refined Ritz vectors
and to retain the genuine ones accurately.
We will borrow the approach of the first author in \cite{jia_multiple} that deals
with a multiple eigenvalue and determines the multiplicity and
corresponding eigenspace, and
propose an approach to accurately removing extra refined Ritz vectors, i.e.,
spurious Ritz values.

Before describing the removal approach, we first introduce some notations.
Given a Ritz pair $(\tilde\lambda, \tilde x)$ and the corresponding
refined Ritz vector $\hat x$, we compute
\begin{equation} \label{eq: refined Ritz value and residual}
  \hat{\lambda} = \hat{x}^H A \hat{x}, \quad
  \hat{r} = A \hat{x} - \hat{\lambda} \hat{x}.
\end{equation}
Note that $\tilde\lambda=\tilde{x}^HA\tilde{x}$.
Since the refined Ritz value $\hat{\lambda}$ minimizes the residual norm $\norm{A \hat{x} - \mu \hat{x}}$
over $\mu\in \mathbb{C}$ and $\hat{x}$ is generally more accurate than
$\tilde{x}$, the refined Ritz value
$\hat\lambda$ is generally more accurate than $\tilde\lambda$ as an approximation to
$\lambda$. On the other hand, since $\hat{x}$ satisfies the residual optimality \eqref{eq: def: refined Rayleigh--Ritz projection}, we have
\begin{displaymath}
  \norm{\hat{r}}
  =    \norm{(A -   \hat{\lambda} I)   \hat{x}}
  \leq \norm{(A - \tilde{\lambda} I)   \hat{x}}
  \leq \norm{(A - \tilde{\lambda} I) \tilde{x}}.
\end{displaymath}
This indicates that the pair $(\hat{\lambda}, \hat{x})$ is generally more accurate than the pairs
$(\tilde{\lambda}, \hat{x})$ and $(\tilde{\lambda}, \tilde{x})$.
%Since the refined Ritz vector $\hat{x}$ is always a good approximation to the eigenvector $x$,
%the Rayleigh quotient $\hat{\lambda}$ is also a good approximation to the associated eigenvalue.
We call $(\hat{\lambda}, \hat{x})$ the refined Ritz pair.

For the diagonalizable matrix $A$,
let $\lambda_1, \ldots, \lambda_m$ be its distinct eigenvalues with multiplicities $n_1, \ldots, n_m$, respectively, and $P_i$ be the spectral projector corresponding to $\lambda_i$.
%Now we describe our removal approach.
Suppose that there is a cluster of $\hat{n}_i$ nearby refined Ritz values $\hat{\lambda}_{i1}, \hat{\lambda}_{i2}, \ldots, \hat{\lambda}_{i\hat{n}_i}$ in the region, which
approximate $\lambda_i$ in the region, and the corresponding refined Ritz vectors are
$\hat{x}_{i1}, \hat{x}_{i2}, \ldots, \hat{x}_{i\hat{n}_i}$.
Let $x_{ij} = P_i \hat{x}_{ij} / \norm{P_i \hat{x}_{ij}}$, $j=1, 2, \ldots, \hat{n}_i$,
be the normalized projection of $\hat{x}_{ij}$ onto the eigenspace associated with $\lambda_i$.
Then $x_{ij}$'s are unit-length eigenvectors of $A$ associated with $\lambda_i$.

First of all, we need to quantitatively determine a cluster of nearby
refined Ritz values approximating the same eigenvalue $\lambda$ in the region.
%A refined Ritz value $\hat{\lambda}$ is said to approximate an eigenvalue $\lambda$ if $\lambda$ is the
%closest eigenvalue to $\hat{\lambda}$.
Since $\lambda$ is unknown in advance, we cannot explicitly exploit it to carry out this task.
The following result provide us an effective approach for this purpose.

\begin{theorem} \label{thm5.1}
   Let $A = X \Lambda X^{-1}$ be the eigendecomposition of $A$. If  $\hat{\lambda}_i$ and $\hat{\lambda}_j$
   are in the region and
   \begin{equation} \label{eq: criterion for clustering refined Ritz values when A diagonalizable}
    \abs{\hat{\lambda}_i - \hat{\lambda}_j}
    \leq \kappa(X) \left( \norm{\hat{r}_i} + \norm{\hat{r}_j} \right),
    \end{equation}
  then $\hat{\lambda}_i$ and $\hat{\lambda}_j$ approximate the same eigenvalue $\lambda_*$ of $A$ in the region
  and belong to the same cluster
  for sufficiently small $\norm{\hat{r}_i}$ and $\norm{\hat{r}_j}$.
  For $A$ Hermitian, \eqref{eq: criterion for clustering refined Ritz values when A diagonalizable} becomes
  \begin{equation} \label{eq: criterion for clustering refined Ritz values when A Hermitian}
    \abs{\hat{\lambda}_i - \hat{\lambda}_j}
    \leq \norm{\hat{r}_i} + \norm{\hat{r}_j}.
  \end{equation}
\end{theorem}

\begin{proof}
The Bauer--Fike theorem (cf. \cite[Theorem 7.2.2]{Golub_Matrix}) states that for any eigenvalue $\mu$ of a perturbed matrix $A + E$ we have
  \begin{displaymath}
    \min_{\lambda \in \lambda(A)} \abs{\mu - \lambda}
    \leq \kappa(X) \norm{E}.
  \end{displaymath}
  Since $\hat{\lambda}_i$ and $\hat{\lambda}_j$ are the eigenvalues of perturbed
  $A - \hat{r}_i \hat{x}_i^H$
  and $A - \hat{r}_j \hat{x}_j^H$ with $\|\hat{r}_i \hat{x}_i^H\|=\|\hat{r}_i\|$
  and $\|\hat{r}_j \hat{x}_j^H\|=\|\hat{r}_j\|$, respectively, by the above error bound
  and \eqref{eq: criterion for clustering refined Ritz values when A diagonalizable}, there exist two
  eigenvalues $\lambda_*$ and $\lambda_\star$ of $A$ such that
  \begin{eqnarray*}
    \abs{\hat{\lambda}_i - \lambda_*}
    &\leq& \kappa(X) \norm{\hat{r}_i}, \\
    \abs{\hat{\lambda}_j - \lambda_\star}
    &\leq &\kappa(X) \norm{\hat{r}_j}.
 \end{eqnarray*}
Therefore, by them and condition
condition \eqref{eq: criterion for clustering refined Ritz values when A diagonalizable}, we obtain
\begin{eqnarray*}
|\lambda_*-\lambda_\star|&=&|\lambda_*-\hat{\lambda}_i+\hat{\lambda}_i-
(\lambda_\star-\hat{\lambda}_j+\hat{\lambda}_j)|\\
&=&|\lambda_*-\hat{\lambda}_i-(\lambda_\star-\hat{\lambda}_j)+\hat{\lambda}_i-\hat{\lambda}_j|\\
&\leq& |\lambda_*-\hat{\lambda}_i|+|\lambda_\star-\hat{\lambda}_j|+\hat{\lambda}_i-\hat{\lambda}_j|\\
&\leq&2\kappa(X) \left( \norm{\hat{r}_i} + \norm{\hat{r}_j} \right).
\end{eqnarray*}
Let $|\lambda_*-\lambda_\star|=\delta$. Notice that $\delta$ is a fixed constant and is either zero or positive, and $\delta=0$ means $\lambda_*=\lambda_\star$.
Since sufficiently small $\|\hat{r}_i\|$ and $\|\hat{r}_j\|$ can make
the above right-hand side arbitrarily small, we must have $\lambda_*=\lambda_\star$.
\end{proof}

For a non-Hermitian matrix $A$,
suppose that the size of $\kappa(X)$ is modest, say no more than 100,
meaning that each eigenvalue $\lambda$ of $A$ is well conditioned. Then
based on \cref{thm5.1}, we determine a cluster of nearby refined Ritz values approximating the same
eigenvalue $\lambda_*$ by checking if $\hat{\lambda}_i$ and $\hat{\lambda}_j$ satisfy
\cref{eq: criterion for clustering refined Ritz values when A diagonalizable} for a general $A$, once
the relative residual norms $\norm{\hat{r}_i}/\|A\|$ and $\norm{\hat{r}_j}/\|A\|$
become reasonably small, say no more than $10^{-3}$. In computations, $\|A\|$ can be
replaced by cheaply computable $\sqrt{\|A\|_1\|A\|_{\infty}}$.
For $\kappa(X)$ large, $\lambda_*$ may be ill conditioned. Therefore, in order
to reliably determine if $\hat{\lambda}_i$ and $\hat{\lambda}_j$ are in
the same cluster, one needs
smaller $\norm{\hat{r}_i}/\|A\|$ and $\norm{\hat{r}_j}/\|A\|$.

Next we consider how to identify spurious approximate eigenvalues and remove
extra refined Ritz vectors.
An analogous proof to that of \cite[Theorem 5]{jia_multiple} leads to
the following theorem.

\begin{theorem}\label{thm5.2}
  Let $X_i = (x_{i1}, \ldots, x_{i\hat{n}_i})$,
  $\hat{X}_i = (\hat{x}_{i1}, \ldots, \hat{x}_{i\hat{n}_i})$,
  and $\sigma_{\min}(X_i)$ and $\sigma_{\min}(\hat{X}_i)$ denote the smallest singular values of $X_i$
  and $\hat{X}_i$, respectively. Then
  \begin{eqnarray}
    \sigma_{\min}(\hat{X}_i)
    &\leq& \sigma_{\min}(X_i) + \sqrt{\hat{n}_i} \max_{1 \leq j \leq \hat{n}_i} \norm{x_{ij} - \hat{x}_{ij}},\label{eqmin}\\
   \sigma_{\min}(\hat{X}_i)&\geq& \sigma_{\min}(X_i)-\sqrt{\hat{n}_i} \max_{1 \leq j \leq \hat{n}_i} \norm{x_{ij} - \hat{x}_{ij}}.\label{eqmin2}
\end{eqnarray}
For $\hat{n}_i>n_i$, the $\hat{n}_i-n_i$ smallest singular values
$\sigma_j(\hat{X}_i), j=n_i+1,\ldots,\hat{n}_i$ satisfy
\begin{equation} \label{smallsv}
    \begin{aligned}
      \sigma_j(\hat{X}_i)
      \leq & \sqrt{\hat{n}_i}\max_{1 \leq j \leq \hat{n}_i} \norm{x_{ij} - \hat{x}_{ij}}
      = \sqrt{\hat{n}_i} \max_{1 \leq j \leq \hat{n}_i} 2 \sin \frac{\angle(x_{ij}, \hat{x}_{ij})}{2} \\
      & \approx \sqrt{\hat{n}_i} \max_{1 \leq j \leq \hat{n}_i} \sin\angle(x_{ij}, \hat{x}_{ij})
      \qquad \text{for small } \angle(x_{ij}, \hat{x}_{ij}),
    \end{aligned}
  \end{equation}
where $\angle(\cdot, \cdot)$ denotes the acute angle of two nonzero vectors.
\end{theorem}

\begin{remark}\label{rema}
Bounds~\eqref{eqmin} and \eqref{eqmin2} show that $\sigma_{\min}(\hat{X}_i)>0$, i.e., the columns of $\hat{X}_i$ must be
linearly independent, provided that $\sigma_{\min}(X_i)>0$ and the $\hat{x}_{ij}$ are sufficiently
accurate approximations to the $x_{ij}$. Bound~\eqref{smallsv} indicates that
$\hat{X}_i$ must be numerically rank deficient and its numerical rank is $n_i$ when taking
the truncation threshold of TSVD as the right-hand side of the bound.
\end{remark}

Based on \cref{thm5.2},
we are able to determine whether or not $\hat{x}_{i1},
\ldots, \hat{x}_{i\hat{n}_i}$ are nearly linearly independent
and to remove $\hat{n}_i-n_i$ extra refined Ritz vectors whenever
$\hat{n}_i>n_i$.
% In practice,
% as stated in \cite[Section 3]{jia_multiple},
% we regard $\hat{X}_i$ as numerically rank deficient if \cref{eq: upper bound for the smallest singular value of refined Ritz vectors} satisfies
% and thus detect the multiplicity $n_i$ by checking the singular values of $\hat{X}_i$
% and retain only $n_i$ linearly independent refined Ritz vectors.

%For $\sin\angle(x_{ij}, \hat{x}_{ij})$ in \cref{eq: upper bound for the smallest singular value of
%refined Ritz vectors},
%it can be estimated by the residual norm $\norm{\hat{r}_{ij}}$ of the refined Ritz pair
%$(\hat{\lambda}_{ij}, \hat{x}_{ij})$ defined in \eqref{eq: refined Ritz value and residual}.
The first author has established a compact
error bound for an approximate eigenvector in terms of the residual norm of approximate eigenpair
(cf. \cite[Theorem 3]{jia_multiple}),
which applies to our context directly, as shown below, where the only change is that we have
used the identity $\|P_i\|=\|I-P_i\|$ (cf. \cite{szyld2006}) to replace $\|I-P_i\|$ by $\|P_i\|$.

\begin{theorem}\label{eqres}
 Let $(\hat{\lambda}_{ij}, \hat{x}_{ij})$ be an approximate eigenpair of $A$ with $\hat{\lambda}_{ij}$
 approximating $\lambda_i$, and define the gap
  \begin{equation} \label{eq: def: actual gap between eigenvalues and refined Ritz value}
    g_{ij} = \min_{s \neq i} \abs{\lambda_s - \hat{\lambda}_{ij}},
  \end{equation}
  and the matrix
  \begin{displaymath}
    X_{ij} = (P_1 \hat{x}_{ij}, \ldots, P_{i-1} \hat{x}_{ij}, P_{i+1} \hat{x}_{ij}, \ldots, P_m \hat{x}_{ij}).
  \end{displaymath}
  Then
  \begin{equation} \label{eq: upper bound for sin(x, x_hat)}
    \sin \angle(x_{ij}, \hat{x}_{ij})
   \leq \inf_{D \text{ diag.}} \kappa(X_{ij} D) \|P_i\| \cdot \frac{\norm{\hat{r}_{ij}}}{g_{ij}}.
  \end{equation}
  If $A$ is Hermitian, then
  \begin{equation} \label{eq: upper bound for sin(x, x_hat) when A Hermitian}
    \sin \angle(x_{ij}, \hat{x}_{ij})
  \leq \frac{\norm{\hat{r}_{ij}}}{g_{ij}}.
  \end{equation}
\end{theorem}

\begin{remark}
The quantities $\inf_{D \text{ diag.}} \kappa(X_{ij} D)$ and $\|P_i\|$ measure
the degree of linear independence of
the eigenvectors $x_{1j},\ldots,x_{i-1j},x_{i+1j},\ldots,x_{mj}$ of $A$
and the conditioning of $\lambda_i$, respectively.
It is known from \cite[Theorem 7.1 and (7.3)]{jiastewart2001} that
$\|\hat{r}_{ij}\|/\|A\|, \ j=1, \ldots, \hat{n}_i$ are as small as the derivations of $x_{ij}$ from the underlying
subspace $\mathcal{R}(S)$. Therefore, provided that $g_{ij}$ defined by \eqref{eq: def: actual gap between eigenvalues and refined Ritz value} is not very small, i.e., $\lambda_i$ is not very close to
the other distinct eigenvalues of $A$,
bounds~\eqref{eqmin}--\eqref{eqmin2} and \cref{eqres} show that $\sigma_{\min}(\hat{X}_i)$ is not small for $\hat{n}_i\leq n_i$ when the $\|\hat{r}_{ij}\|$ are sufficiently small,
indicating that the refined Ritz values $\hat{\lambda}_{i1}, \hat{\lambda}_{i2}, \ldots, \hat{\lambda}_{i\hat{n}_i}$ are genuine
and the refined Ritz vectors $\hat{x}_{i1}, \hat{x}_{i2}, \ldots, \hat{x}_{i\hat{n}_i}$ are thus retained.

On the other hand, a combination of \eqref{smallsv} and
\eqref{eq: upper bound for sin(x, x_hat)}--\eqref{eq: upper bound for sin(x, x_hat) when A Hermitian} indicate that $\hat{X}_i$ exactly has
$\hat{n}_i-n_i$ small singular values whose sizes are decided by the residual norms and gaps.
We are thus sure that the numerical rank
of $\hat{X}_i$ is $n_i$, and retain $n_i$ refined Ritz values and
remove extra $\hat{n}_i-n_i$ refined Ritz vectors.
\end{remark}

\begin{remark}
The gap $g_{ij}$ in \eqref{eq: def: actual gap between eigenvalues and refined Ritz value}
can be estimated by the distances between the approximating
$\hat{\lambda}_{ij}$ and other refined Ritz values: Suppose that for some $s \in \{ 1, 2, \ldots, m \}$,
there is a cluster of $\hat{n}_s$ nearby refined Ritz values $\hat{\lambda}_{s1}, \hat{\lambda}_{s2}, \ldots, \hat{\lambda}_{s\hat{n}_s}$ that approximate $\lambda_s$.
Then we replace $g_{ij}$ by its computable approximation
\begin{equation}\label{tildeg_ij}
\tilde{g}_{ij} = \min_{\substack{s \neq i \\ 1 \leq t \leq \hat{n}_s}} \abs{\hat{\lambda}_{st} - \hat{\lambda}_{ij}}.
\end{equation}
\end{remark}
It follows from \eqref{smallsv} and \eqref{eq: upper bound for sin(x, x_hat)}--\eqref{eq: upper bound for sin(x, x_hat) when A Hermitian} that bound~\eqref{smallsv} is approximately
\begin{equation} \label{eq: criterion for removing refined Ritz vectors}
\sqrt{\hat{n}_i} \max_{1 \leq j \leq \hat{n}_i} C_{ij} \frac{\norm{\hat{r}_{ij}}}{\tilde{g}_{ij}},
\end{equation}
where $C_{ij} = \inf_{D \text{ diag.}} \kappa(X_{ij} D) \|P_i\|$ for a
general diagonalizable $A$ and $C_{ij} = 1$ for a Hermitian $A$. In computations,
the $C_{ij}$ are unknown for a general $A$, and we set them to be no more than some modest constant,
say $100$, which amount to supposing that the eigenvectors of $A$ are not ill conditioned.
% Replace $g_{ij}$ in \eqref{eq: def: actual gap between eigenvalues and refined Ritz value} by $\tilde{g}_{ij}$ in \eqref{tildeg_ij}.
We compute the SVD of $\hat{X}_i$ and check how many of its singular values meet \cref{eq: criterion for removing refined Ritz vectors}: If none satisfies the bound,
the $\hat{n}_i$ refined Ritz values $\hat{\lambda}_{ij}$
are genuine, and we retain the $\hat{n}_i$ refined Ritz vectors; if
$\hat{n}_i-n_i$ ones are no more than the small quantity in
\cref{eq: criterion for removing refined Ritz vectors},
there are $\hat{n}_i-n_i$
spurious refined Ritz values $\hat{\lambda}_{ij}$, and we remove any
$\hat{n}_i-n_i$ extra refined Ritz vectors from $\hat{X}_i$. It is worth pointing out that
the size of $\max_{1 \leq j \leq \hat{n}_i} C_{ij}$ is nonessential and we only require
that the $\|\hat{r}_{ij}\|$ are such that the quantity in \cref{eq: criterion for removing refined Ritz vectors} is small: the larger $\max_{1 \leq j \leq \hat{n}_i} C_{ij}$ is, the smaller the $\|\hat{r}_{ij}\|$ are required.

\begin{remark}
Computationally, the aforementioned removal approach can be carried out without
computing the singular values of the overly tall $\hat{X}_i$.
Relation~\eqref{eq: refined Rayleigh--Ritz projection in matrix form} shows
that $\hat{x}_{ij} = U \hat{z}_{ij}$. We call the $\hat{z}_{ij}$ primitive refined Ritz vectors.
Since $U$ is column orthonormal, the singular values of  $\hat{X}_i$ and  $\hat{Z}_i=(\hat{z}_{i1}, \ldots, \hat{z}_{i\hat{n}_i})$ are identical.
Therefore, we only need to compute
the SVD of the small sized matrix $\hat{Z}_i$ and determine its numerical rank $n_i$.
\end{remark}
% \cref{alg: removal approach} describes our
% removal approach of extra refined Ritz vectors.

% \begin{algorithm}[htbp]
% \caption{Removing extra refined Ritz vectors from a cluster}
% \label{alg: removal approach}
% \begin{algorithmic}[1]
%   \REQUIRE $s$ primitive refined Ritz vectors $\hat{z}_{1}, \ldots, \hat{z}_s$ aligned to the eigenspace of $\lambda$.
%   \STATE Set $\hat{Z} = [\hat{z}_1]$.
%   \FOR {$j = 2, 3, \ldots, s$}
%     \STATE Set $\hat{Z} = [\hat{Z}, \hat{z}_j]$.
%     \IF{$\sigma_{\min}(\hat{Z})$ satisfies \cref{eq: criterion for removing refined Ritz vectors}}
%       \STATE Discard $\hat{z}_j$.
%     \ELSE
%       \STATE Go to step 3.
%     \ENDIF
%   \ENDFOR
%   \RETURN The retained refined Ritz vectors in $\hat{X}$.
% \end{algorithmic}
% \end{algorithm}

% \textcolor{red}{As an alternative of \cref{alg: removal approach}, it is simpler to
% compute the SVD of $\hat{Z}_i=(\hat{z}_1, \ldots, \hat{z}_s)$ once and
% determine its numerical rank and spurious
% refined Ritz values and refined Ritz vectors. This alternative is less costly
% than \cref{alg: removal approach}, which computes the SVDs of
% a sequence of augmenting matrices,
% unless $\hat{n}_i\gg n_i$.}

In summary, we describe our removal approach as \cref{alg: removal}, called the \emph{refined removal approach}. The refined removal approach is tune free and does not need any
auxiliary user-specified tolerance, unlike that in a SS-RR algorithm, and it is
automatically performed during a SS--RRR algorithm, especially for $A$ Hermitian.
Thus the approach is easy to use in practice.

\begin{algorithm}[htbp]
\caption{Refined removal approach}
\label{alg: removal}
\begin{algorithmic}[1]
    \STATE
    Cluster the refined Ritz values based on
    \cref{eq: criterion for clustering refined Ritz values when A diagonalizable}
    or \cref{eq: criterion for clustering refined Ritz values when A Hermitian}.

 \STATE For each cluster, compute the SVD of $\hat{Z}_i$, check how many of its singular values
    are no more than the quantity \cref{eq: criterion for removing refined Ritz vectors};
    remove the extra refined Ritz vectors if any, and retain the refined Ritz
    values and vectors whose numbers are the numerical rank of $\hat{Z}_i$.
\end{algorithmic}
\end{algorithm}

\subsection{A complete algorithm}

Now we focus on the case that $A$ is Hermitian.
Combining the refined Rayleigh--Ritz projection \cref{eq: def: refined Rayleigh--Ritz projection}
with the refined removal approach, %or the global refined removal approach,
we describe a restarted CJ--SS--RRR algorithm as \cref{alg: CJSSRRR}.
The differences between \cref{alg: CJSSRR} and \cref{alg: CJSSRRR} are:
(i) the Rayleigh--Ritz projection in \cref{alg: CJSSRR} is replaced by the refined Rayleigh--Ritz projection in steps \ref{step: CJSSRR: RR}--\ref{step: CJSSRRR: refined} of \cref{alg: CJSSRRR};
(ii) in step 8 of \cref{alg: CJSSRRR},
the spurious refined Ritz values and
extra refined Ritz vectors are removed by the refined removal approach.
For the convergence test in step \ref{step: CJSSRRR: convergence test}, a refined Ritz pair is
claimed to have
converged if its relative residual norm drops below a given tolerance $tol$:
\begin{displaymath}
  \frac{\norm{\hat{r}_i^{(k)}}}{\norm{A}} < tol,
\end{displaymath}
where $\norm{A} = \max \{ \abs{\lambda_{\min}}, \abs{\lambda_{\max}} \}$.
%Since the CJ--SS--RR method is based on Rayleigh--Ritz projection and \cref{alg: CJSSRRR} is based on refined %Rayleigh--Ritz projection, \cref{alg: CJSSRRR} can also be referred as the CJ--SS--RRR algorithm.

\begin{algorithm}[htbp]
\caption{The restarted CJ--SS--RRR algorithm}
\label{alg: CJSSRRR}
\begin{algorithmic}[1]
  \REQUIRE The Hermitian matrix $A$, the real interval $[a, b]$, the series degree $d$, the moment number $M$, an $n$-by-$\ell$ starting matrix $V=V^{(0)}$ with $M\ell \geq n_{ev}$, and the stopping tolerance $tol$.
  \STATE Compute the CJ coefficients by \cref{eq: def: Chebyshev coefficient,eq: def: Jackson damping factor}.
  \FOR {$k = 1, 2, \ldots$}
    \STATE Compute $S^{(k)}=(S_0^{(k)},S_1^{(k)},\ldots,S_k^{(k)})$ by \cref{eq: def: polynomial approximation of S}. \label{step: CJSSRRR: compute S}
    \STATE Compute the QR factorization: $S^{(k)} = U^{(k)}R^{(k)}$.
    \STATE Apply the Rayleigh--Ritz projection \cref{eq: def: Rayleigh--Ritz procedure} to  $\mathcal{R}(U^{(k)})$, and
     compute the Ritz values $\tilde{\lambda}_i^{(k)}, i = 1, 2, \ldots, M\ell$. \label{step: CJSSRRR: RR}
    \STATE Select the potential Ritz values $\tilde{\lambda}_1^{(k)}, \ldots, \tilde{\lambda}_p^{(k)}$ in $[a, b]$ and
    compute their corresponding refined Ritz vectors $\hat{x}_1^{(k)}, \ldots, \hat{x}_p^{(k)}$ by \cref{eq: refined Rayleigh--Ritz projection with QR}.\label{step: CJSSRRR: refined}
    \STATE Compute the refined Ritz values $\hat{\lambda}_1^{(k)}, \ldots, \hat{\lambda}_p^{(k)}$ and residuals $\hat{r}_1^{(k)}, \ldots, \hat{r}_p^{(k)}$ by \cref{eq: refined Ritz value and residual}.
    \STATE For $\hat{\lambda}_1^{(k)}, \ldots, \hat{\lambda}_p^{(k)}$
    and  $\hat{x}_1^{(k)}, \ldots, \hat{x}_p^{(k)}$, perform \cref{alg: removal}, and
    only retain the genuine refined Ritz values and vectors.
    \STATE Test the convergence of the retained refined Ritz pairs. \label{step: CJSSRRR: convergence test}
    \STATE If not converged, set the updated starting
    $V^{(k+1)}$ to be the first $\ell$ columns of $S^{(k)}$ and restart.
  \ENDFOR
  \RETURN The converged $n_{ev}$ refined Ritz pairs $(\hat{\lambda}_i^{(k)}, \hat{x}_i^{(k)})$.
\end{algorithmic}
\end{algorithm}

%For a general diagonalizable matrix $A$,
%we can restart a SS--RRR algorithm similarly.
%based on
%the refined Rayleigh--Ritz projection and \cref{alg: removal approach} or \cref{alg: removal approach}b
%by replacing step \ref{step: CJSSRRR: compute S} in \cref{alg: CJSSRRR} with the step for computing
%$S^{(k)}$ in the SS--RR method \cite{block_SS_RR,SS_RR},
%and replacing the criterion
%\cref{eq: criterion for clustering refined Ritz values when A Hermitian}
%in step \ref{step: CJSSRRR: clustering} with
%\cref{eq: criterion for clustering refined Ritz values when A diagonalizable}.
%However, in this case,
%the constants $\kappa(X)$ in
%\cref{eq: criterion for clustering refined Ritz values when A diagonalizable}
%and $C_{ij}$ in
%\cref{eq: criterion for removing refined Ritz vectors}
%are unknown and a priori.
%Jia \cite{jia_multiple} suggests setting them to some modest constant not exceeding 1000, which amounts to
%assuming that the eigenvalues $\lambda$ are not too ill conditioned.

\section{Numerical experiments}\label{sec:6}

In this section, we present numerical experiments to illustrate the performance of \cref{alg: CJSSRRR} and
its superiority to \cref{alg: CJSSRR}. We compare the refined removal approach with the
TSVD approach and
the residual norm-based removal approach used in \cref{alg: CJSSRR}, and show that the former one is much
more reliable and robust than the latter two ones.
The test matrices are from the SuiteSparse Matrix Collection \cite{matrix_collection}, and we list some of their basic properties and the intervals of interest in \cref{tab: test matrices}.
For each matrix, interval I is close to one end of the spectrum, while
interval II is inside the spectrum.
The minimum and maximum eigenvalues $\lambda_{\min}, \lambda_{\max}$ of each matrix are computed by {\sc
Matlab}'s \texttt{eigs} function and the number $n_{ev}$ of desired eigenvalues  in each interval is computed by the CJ-FEAST eigensolver, which is a direct adaptation of the FEAST SVDsolvers in \cite{CJ_FEAST_cross,CJ_FEAST_augmented}.
All the numerical experiments were performed on an Intel Core i7-9700, CPU 3.0GHz, 8GB RAM
using {\sc Matlab} R2024b with the machine precision $\epsilon_{\rm mach} = 2.22\times 10^{-16}$ under the Microsoft Windows 10 64-bit system. The unit-length desired eigenvectors $x_1, \ldots, x_{n_{ev}}$ are computed by the CJ-FEAST eigensolver too, and their relative residual norms are smaller than $10^{-14}$, the level of $\epsilon_{\rm mach}$. We thus
regard the computed solutions by the CJ-FEAST eigensolver as ``exact'' ones, and use them as a reference
standard.

\begin{table}[htbp]
{\small\caption{Properties of the test matrices and the intervals of interest.}
\label{tab: test matrices}
\centering
\begin{tabular}{*{8}{c}}
\toprule
matrix            & size  & $\lambda_{\min}$ & $\lambda_{\max}$ & Interval I  & $n_{ev}$ & Interval II  & $n_{ev}$ \\
\midrule
SiH4              & 5041  & $-0.996$           & 36.8             & $[0.5,2.5]$   & 120      & $[17.5,18]$   & 123      \\
SiNa              & 5743  & $-0.705$           & 25.7             & $[0,2]$       & 178      & $[12.5,13]$   & 183      \\
delaunay\_n13     & 8192  &$-3.570$            & 6.46             & $[5.2,6.2]$   & 310      & $[1.5,2]$     & 312      \\
benzene           & 8219  & $-0.730$           & 58.4             & $[1.5,4.5]$   & 229      & $[28,29]$     & 231      \\
stokes64          & 12546 & $-1.70\times 10^{-3}$         & 4.00             & $[3.8,3.9]$   & 232      & $[1.9,2.1]$   & 232      \\
Pres\_Poisson     & 14822 &  $1.28\times 10^{-5}$         & 26.1             & $[8,13]$      & 68       & $[4.2,5]$     & 66       \\
Si10H16           & 17077 & $-1.150$            & 37.0             & $[0.5,2]$     & 307      & $[18,18.4]$   & 303      \\
% brainpc2          & 27607 & -2000            & 4465             & [1905,1915] & 146      & [-30,0]     & 148      \\
% rgg\_n\_2\_15\_s0 & 32768 & -5.12            & 17.4             & [11,13]     & 283      & [5.5,6]     & 286      \\
% SiO               & 33401 & -1.68            & 84.4             & [2,4]       & 332      & [41,41.5]   & 332      \\
% Andrews           & 60000 & -2.17e-14        & 36.5             & [1.5,2.5]   & 346      & [18.1,18.3] & 353      \\
\bottomrule
\end{tabular}}
\end{table}

%In what follows, since the CJ series expansion approximation
%requires the spectrum of $A$ to be in $[-1, 1]$, the matrix $A$ and
%interval ends $a, b$ are replaced by $l(A)$, $l(a)$, $l(b)$, respectively,
%where the linear transformation
%$l(t) = 2(t - \lambda_{\min})/(\lambda_{\max} - \lambda_{\min}) - 1$.
% To make a fair comparison, for each test problem and given $\ell$, we used the same starting $n \times \ell$ matrix $V^{(0)}$ for the algorithms tested, which was generated randomly in a normal distribution.

% \subsection{The performance of the}

\subsection{The performance of the refined removal approach}
\label{subsec: performance of refined removal}

In this subsection, we report the performance of the refined removal approach, i.e., \cref{alg: removal}.
%Before presenting the numerical results, we should point out that the performance of any removal approach is %dependent on the angle between the desired eigenvectors
%and the search subspace $\Range(S)$.
%That is, if one of the desired eigenvectors is orthogonal to $\Range(S)$, then any attempt to select %$n_{ev}$ desired approximate eigenvectors from $\Range(S)$ will fail, as there indeed at most $n_{ev}-1$ %desired eigenvectors can be approximated by $\Range(S)$.
%For the desired eigenvectors $x_1, \ldots, x_{n_{ev}}$ and the search subspace $\Range(S)$,
% We use the largest distance or deviation of the $n_{ev}$ desired eigenvectors from the subspace
We use the deviation of the desired eigenspace $\mathcal{X}_{n_{ev}} = \mathrm{span}\{x_1, \ldots, x_{n_{ev}}\}$ from the underlying subspace
$\Range(S)$:
\begin{equation}\label{deviation}
\epsilon_{ev} :=  \sin\angle(\Range(S), \mathcal{X}_{n_{ev}})
= \norm{(I - UU^H) X_{n_{ev}}}
\end{equation}
to measure the accuracy of $\Range(S)$,
%when extracting approximations from it to the desired eigenspace $\mathcal{X}_{n_{ev}}$,
where the columns of $U$ and $X_{n_{ev}}$ are orthonormal bases of
$\Range(S)$ and $\mathcal{X}_{n_{ev}}$, respectively.

We illustrate the performance of \cref{alg: CJSSRRR} with the refined removal approach for
different $\epsilon_{ev}$.
We took the series degree $d=\lceil \pi^2(b-a)^{-4/3} \rceil - 2$, the moment number $M$ as $8$, and
the dimension $M\ell$ of $\Range(S)$ as $8 \lceil 1.5 n_{ev} / 8 \rceil$ such that $M\ell \geq
1.5 n_{ev}$. For each test matrix and interval in \cref{tab: test matrices},
% we run ten iterations of \cref{alg: CJSSRRR} using five starting matrices $V^{(0)}$ randomly generated in a normal distribution.
we ran \cref{alg: CJSSRRR} five times using five different starting matrices $V^{(0)}$ generated
randomly in the standard normal distribution, and performed
ten restarts in each run.

Denote by $n_{in}$ the number of the refined Ritz values after the removal.
\cref{fig: performance of refined removal: nev} shows the numbers of the retained refined Ritz vectors using $\Range(S)$ for different $\epsilon_{ev}$.
For each test matrix and interval in \cref{tab: test matrices}, the results of
five runs are plotted in the same figure.
We can observe that
% when the accuracy of the search subspace $\Range(S)$ is higher than $10^{-3}$,
once $\epsilon_{ev} \leq 10^{-3}$,
the refined removal approach always retained $n_{ev}$ refined Ritz vectors exactly.
As a matter of fact, we can see that,
except for the test problems SiNa with interval I and delaunay\_n13 with interval II,
the refined removal approach correctly retained the $n_{ev}$ refined Ritz vectors
% when the accuracy of the search subspace $\Range(S)$ is higher than $10^{-2}$.
when $\epsilon_{ev} \leq 10^{-2}$.

\begin{figure}[htbp]
  \centering
  \newcommand{\txtW}{0.47\textwidth}

  \begin{subfigure}[t]{\textwidth}
    \centering
    \includegraphics[width=\txtW]{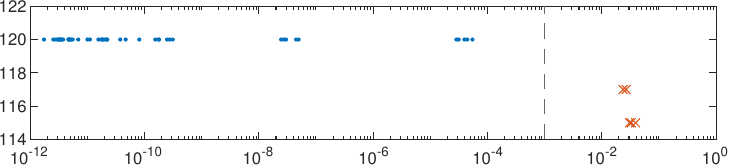}
    \includegraphics[width=\txtW]{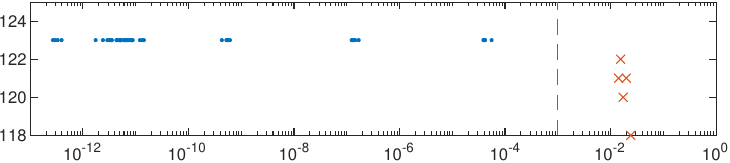}
    \subcaption{SiH4.}
  \end{subfigure}

  \begin{subfigure}[t]{\textwidth}
    \centering
    \includegraphics[width=\txtW]{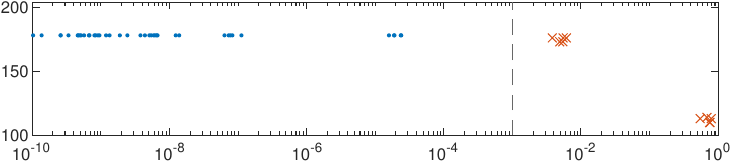}
    \includegraphics[width=\txtW]{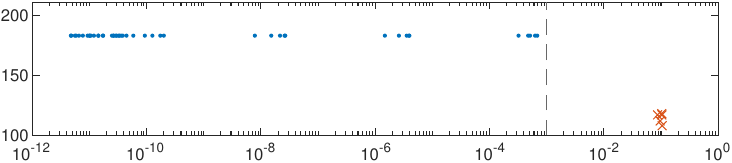}
    \subcaption{SiNa.}
  \end{subfigure}

  \begin{subfigure}[t]{\textwidth}
    \centering
    \includegraphics[width=\txtW]{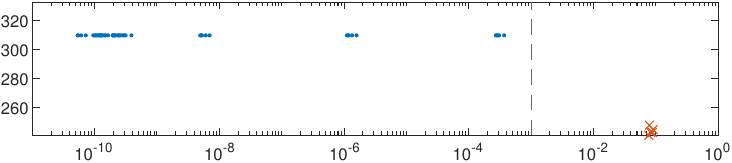}
    \includegraphics[width=\txtW]{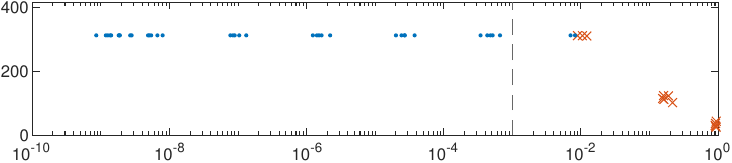}
    \subcaption{delaunay\_n13.}
  \end{subfigure}

  \begin{subfigure}[t]{\textwidth}
    \centering
    \includegraphics[width=\txtW]{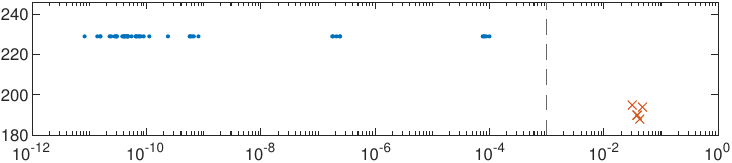}
    \includegraphics[width=\txtW]{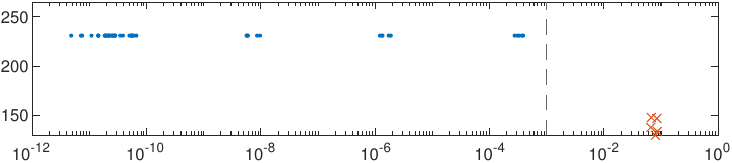}
    \subcaption{benzene.}
  \end{subfigure}

  \begin{subfigure}[t]{\textwidth}
    \centering
    \includegraphics[width=\txtW]{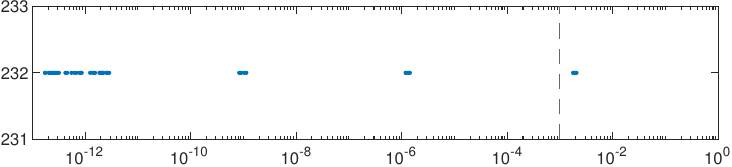}
    \includegraphics[width=\txtW]{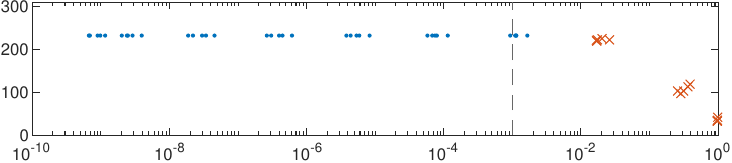}
    \subcaption{stokes64.}
  \end{subfigure}

  \begin{subfigure}[t]{\textwidth}
    \centering
    \includegraphics[width=\txtW]{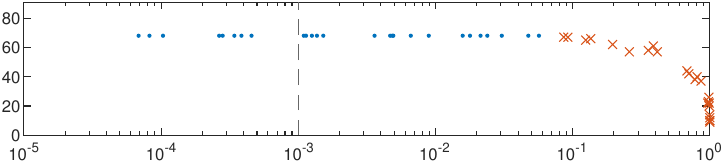}
    \includegraphics[width=\txtW]{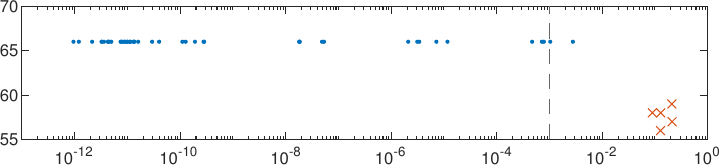}
    \subcaption{Pres\_Poisson.}
  \end{subfigure}

  \begin{subfigure}[t]{\textwidth}
    \centering
    \includegraphics[width=\txtW]{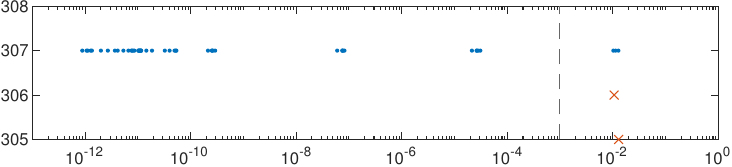}
    \includegraphics[width=\txtW]{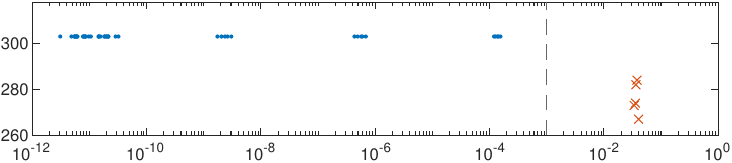}
    \subcaption{Si10H16.}
  \end{subfigure}

  \caption{
    The performance of \cref{alg: CJSSRRR}, where the left and right figures are for each test matrix
    with intervals I and II in \cref{tab: test matrices}, respectively.
    The horizontal axis is the deviation $\epsilon_{ev}$ in \cref{deviation}, and
    the vertical axis is the number ${n}_{in}$ of the retained refined Ritz values after the removal.
    The vertical dashed line is $\epsilon_{ev} = 10^{-3}$.
    For 5 runs, after 10 restarts were performed,
    total 50 points $(\epsilon_{ev}, {n}_{in})$'s for each test problem are plotted in one figure.
    A point $(\epsilon_{ev}, {n}_{in})$ is marked by a dot if ${n}_{in} = n_{ev}$ and by a cross if ${n}_{in} \neq n_{ev}$. Therefore, at each restart, five
    points look like a point.
  }
  \label{fig: performance of refined removal: nev}
\end{figure}

\begin{figure}[htbp]
  \centering
  \newcommand{\txtW}{0.47\textwidth}
  \begin{subfigure}[t]{\txtW}
    \includegraphics[width=\textwidth]{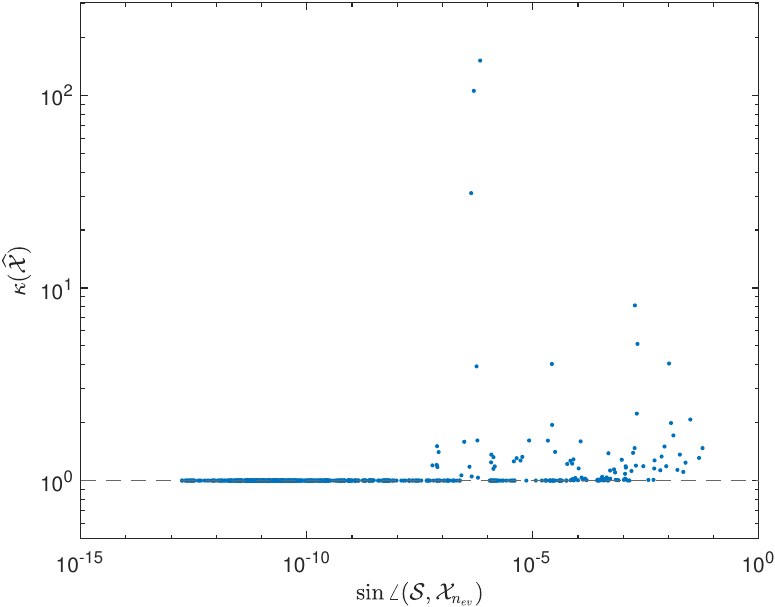}
    \subcaption{The condition number $\kappa(\widehat{X})$.}
    \label{subfig: performance of refined removal: cond}
  \end{subfigure}
  \begin{subfigure}[t]{\txtW}
    \includegraphics[width=\textwidth]{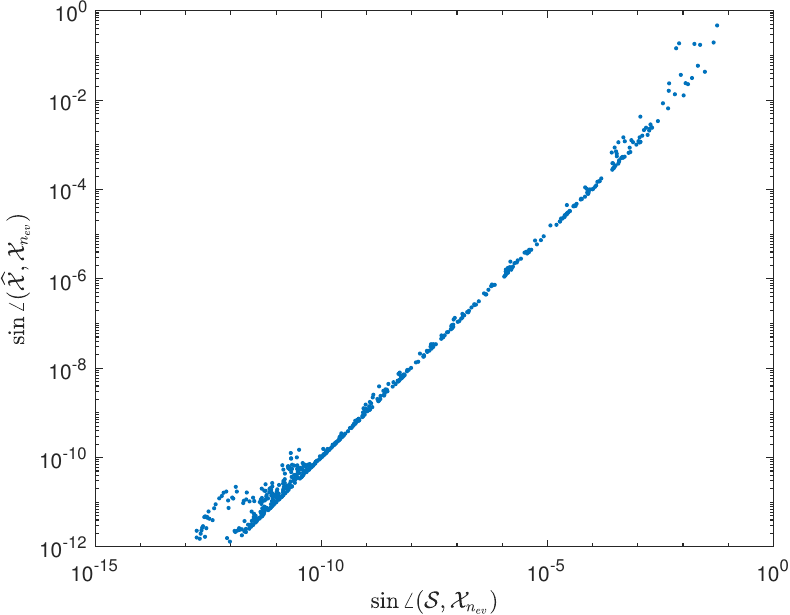}
    \subcaption{The deviation $ \sin\angle(\Range(\widehat{X}), \mathcal{X}_{n_{ev}})$.}
    \label{subfig: performance of refined removal: sin}
  \end{subfigure}

  \caption{
    The performance of \cref{alg: CJSSRRR}.
%    when the number of the retained refined Ritz vectors is equal to $n_{ev}$.
   (a): the condition number $\kappa(\widehat{X})$ of $\widehat{X}$ whose columns
   are the retained refined Ritz vectors.
    (b): the error $\sin\angle(\Range(\widehat{X}), \mathcal{X}_{n_{ev}})$ versus the deviation
    $\epsilon_{n_{ev}}$.
    In each figure,
    the results of all the test matrices and intervals in \cref{tab: test matrices} are plotted together.
  }
  \label{fig: performance of refined removal: cond and sin}
\end{figure}

To demonstrate that the $n_{ev}$ retained refined Ritz vectors $\hat{x}_1, \ldots, \hat{x}_{n_{ev}}$ are
good approximations to the desired eigenvectors, we draw the condition numbers $\kappa(\widehat{X})$ of
$\widehat{X} = (\hat{x}_1, \ldots, \hat{x}_{n_{ev}})$ and the errors
$\sin\angle(\Range(\widehat{X}), \mathcal{X}_{n_{ev}})$
in \cref{fig: performance of refined removal: cond and sin}.
The former measures the linear independence of the retained refined Ritz vectors, while the latter measures the accuracy of $\Range(\widehat{X})$.

In \cref{subfig: performance of refined removal: cond}, the condition numbers $\kappa(\widehat{X})$ are all smaller than $200$ and most of them are very close to one as $\epsilon_{ev}$ becomes small.
Thus the $n_{ev}$ refined Ritz vectors retained by our removal approach are not only indeed linearly
independent but also numerically orthonormal.
This justifies that \cref{alg: removal} is reliable and robust.

We can see from \cref{subfig: performance of refined removal: sin} that as the accuracy of
$\Range(S)$ improves, $\Range(\widehat{X})$ becomes equally
more accurate:
two errors pairwise are almost equal since they are almost $45^{\circ}$ angles.
This indicates that the retained $n_{ev}$ refined Ritz vectors
are almost best approximations to the desired eigenvectors since their errors as small as
$\epsilon_{ev}$ in \eqref{deviation}.
%, which confirm the convergence result on
%refined Ritz vectors in \cite[Theorem 7.1]{jiastewart2001}.

% We have also tested \cref{alg: CJSSRRR}b, i.e., \cref{alg: CJSSRRR} with the removal approach
% \cref{alg: removal approach}b used, and found that \cref{alg: CJSSRRR} and \cref{alg: CJSSRRR}b
% behaved completely the same, meaning that  \cref{alg: removal approach}
% and \cref{alg: removal approach} are equally
% effective and reliably when using them in \cref{alg: CJSSRRR}.
%Once $\epsilon_{ev} \leq 10^{-3}$,
%the removal approaches always selected the numbers $n_{ev}$
%of linearly independent refined Ritz vectors incorrectly,
%and these selected linearly dependent refined Ritz vectors are almost best approximations to the desired
%eigenvectors.

\subsection{The comparison of different removal approaches}

In this subsection, we compare the refined removal approach
with two removal approaches used in the SS--RR methods:
the TSVD approach and the residual norm-based removal approach,
where in the first approach we set the truncation tolerance to be $10^{-12}$ and $10^{-14}$, respectively,
and in the second approach we set the threshold $\delta$ to be $10^{-2}$ and $10^{-4}$, respectively.
A basic property is:
The smaller the tolerance is, the more Ritz values in the given interval may be retained; the smaller
$\delta$ is, the more Ritz values in the interval may be removed.

In the experiments, we ran \cref{alg: CJSSRR} using the above two removal
approaches. Throughout this subsection, we took the same series degree $d$,
moment number $M$, dimension $M\ell$ of $\Range(S)$,
and starting matrix $V^{(0)}$
as those in \cref{subsec: performance of refined removal}.

\begin{figure}[htbp]
  \centering
  \newcommand{\txtW}{0.47\textwidth}

  \begin{subfigure}[t]{\txtW}
    \includegraphics[width=.986\textwidth]{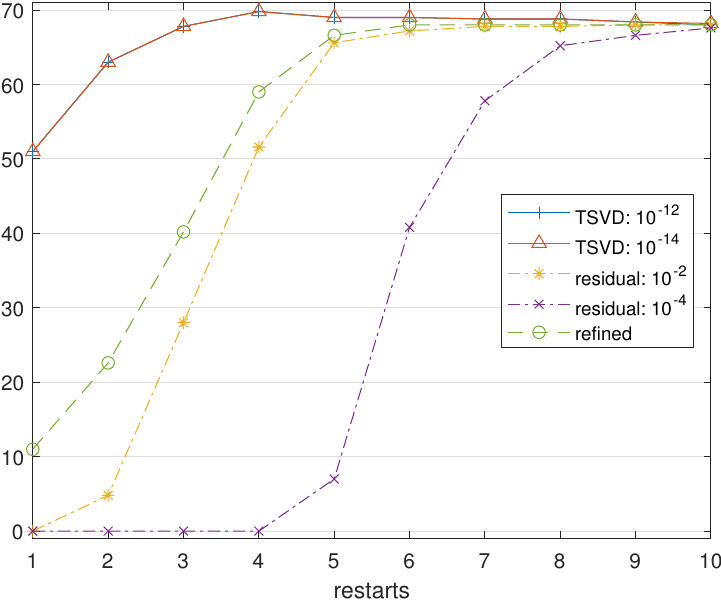}
    \subcaption{Pres\_Poisson with interval I.}
    \label{subfig: average number of different approaches: Pres_Poisson I}
  \end{subfigure}
  \begin{subfigure}[t]{\txtW}
    \includegraphics[width=\textwidth]{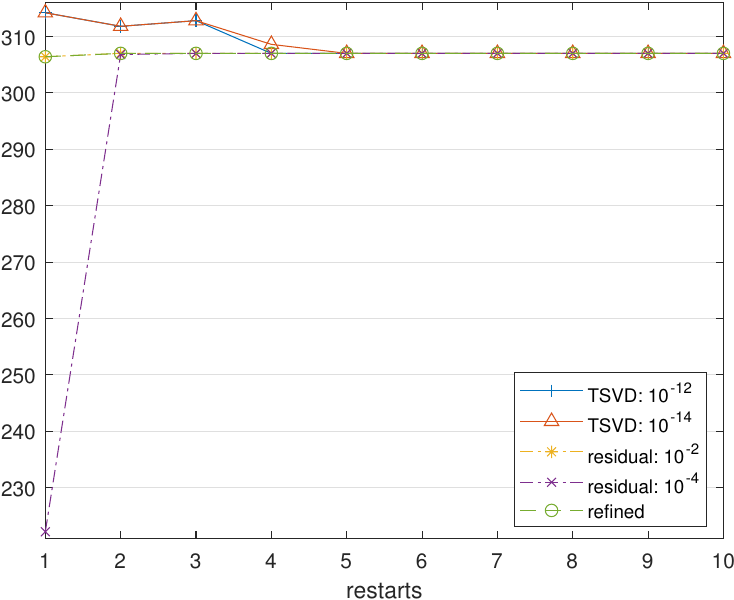}
    \subcaption{Si10H16 with interval I.}
    \label{subfig: average number of different approaches: Si10H16 I}
  \end{subfigure}

  \begin{subfigure}[t]{\txtW}
    \includegraphics[width=\textwidth]{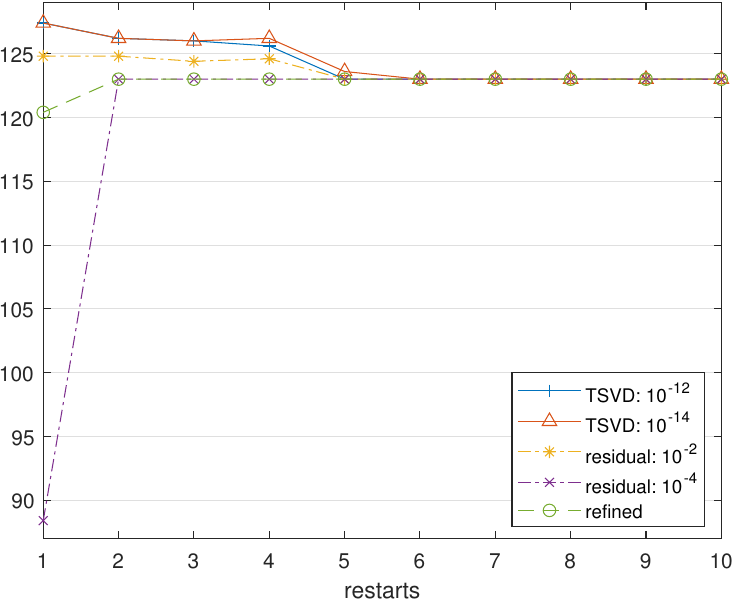}
    \subcaption{SiH4 with interval II.}
    \label{subfig: average number of different approaches: SiH4 II}
  \end{subfigure}
  \begin{subfigure}[t]{\txtW}
    \includegraphics[width=\textwidth]{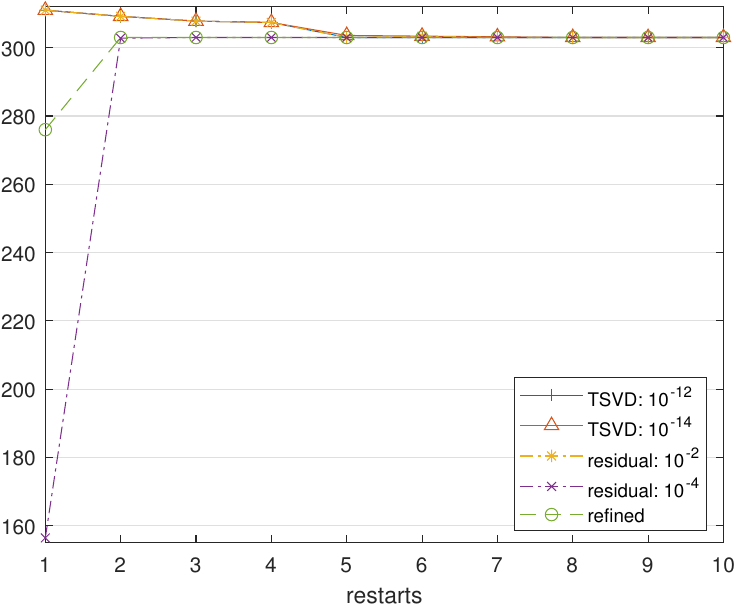}
    \subcaption{Si10H16 with interval II.}
    \label{subfig: average number of different approaches: Si10H16 II}
  \end{subfigure}

  \caption{
    The average numbers of the retained approximate eigenvalues using different removal approaches.
    The label ``TSVD'' denotes the TSVD approach, ``residual'' denotes the residual norm-based removal approach, and ``refined'' denotes the removal approach used for CJ--SS--RRR.
  }
  \label{fig: average number of different approaches}

\end{figure}

\cref{fig: average number of different approaches} shows the average numbers of the retained Ritz and
refined Ritz values using different removal approaches over five runs for four test problems.
\cref{subfig: average number of different approaches: Pres_Poisson I} is for the
test matrix Pres\_Poisson with interval I. We see from it
that, until the 9th restart, the residual norm-based removal approach lost some genuine Ritz values,
and the smaller the threshold $\delta$ is,
the more it lost. In contrast, the TSVD approach and \cref{alg: removal} worked much better, and they
retain $n_{ev}$ approximate eigenpairs from the sixth restart
onwards.

From \cref{subfig: average number of different approaches: SiH4 II,subfig: average number of different approaches: Si10H16 I,subfig: average number of different approaches: Si10H16 II}
for the test matrices SiH4 and Si10H16,
we observe that the refined removal approach
detected the correct numbers $n_{ev}$'s starting from the second restart, much more early than the TSVD
approach and the residual norm-based removal approach. It is clear that, in the first four restarts,
the CJ--SS--RR computed some Ritz values in the intervals, as the curves of the TSVD approach
indicate. In the meantime, \cref{subfig: average number of different approaches: SiH4 II,subfig: average number of different approaches: Si10H16 II} for the
test matrix SiH4 and Si10H16 with interval II
show that if $\delta$ is not small enough then
the spurious Ritz values may not be completely removed by the residual norm-based removal approach,
causing that more than $n_{ev}$ Ritz values were retained.
These observations are in accordance with our previous analysis:
for the Rayleigh--Ritz projection, the residual norm-based removal approach is sensitive to
$\delta$. The fundamental cause is that the Ritz vectors corresponding to a cluster of
nearby Ritz values may or may not be good approximations to desired eigenvectors
as $\epsilon_{ev}\rightarrow 0$,
an intrinsic uncertainty and deficiency of the Rayleigh--Ritz projection method
\cite{jiastewart2001,Stewart_Eigen}.

To compare these removal approaches more comprehensively, we illustrate their performance
for all the test problems in \cref{tab: test matrices}.
% As a reference, we also present the performance of our refined removal approach.
\cref{fig: nin/nev of different approaches} depicts the values $n_{in} / n_{ev}$ for the three removal
approaches.

\begin{figure}[htbp]
  \centering
  \newcommand{\txtW}{0.47\textwidth}

  \begin{subfigure}[t]{\txtW}
    \includegraphics[width=\textwidth]{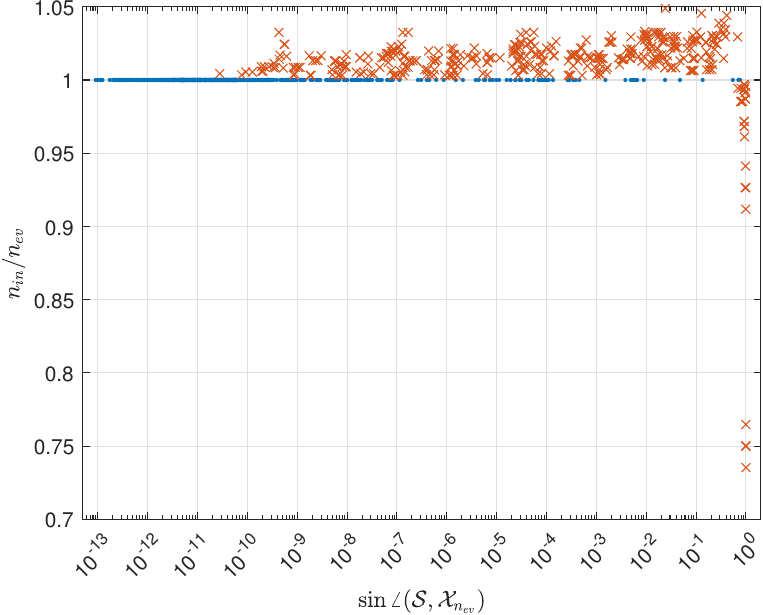}
    \subcaption{TSVD approach with the truncation tolerance $10^{-12}$.}
    \label{subfig: nin/nev of different approaches: TSVD 1e-12}
  \end{subfigure}
  \begin{subfigure}[t]{\txtW}
    \includegraphics[width=\textwidth]{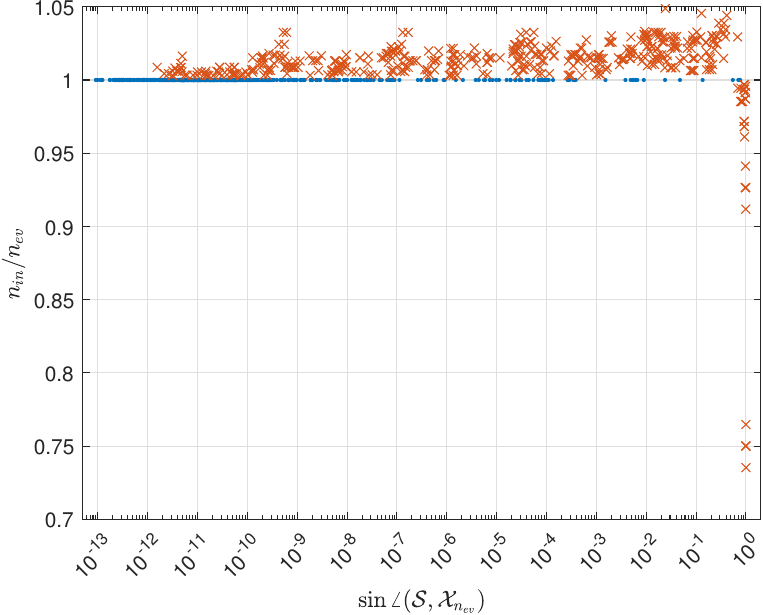}
    \subcaption{TSVD approach with the truncation tolerance $10^{-14}$.}
    \label{subfig: nin/nev of different approaches: TSVD 1e-14}
  \end{subfigure}

  \begin{subfigure}[t]{\txtW}
    \includegraphics[width=\textwidth]{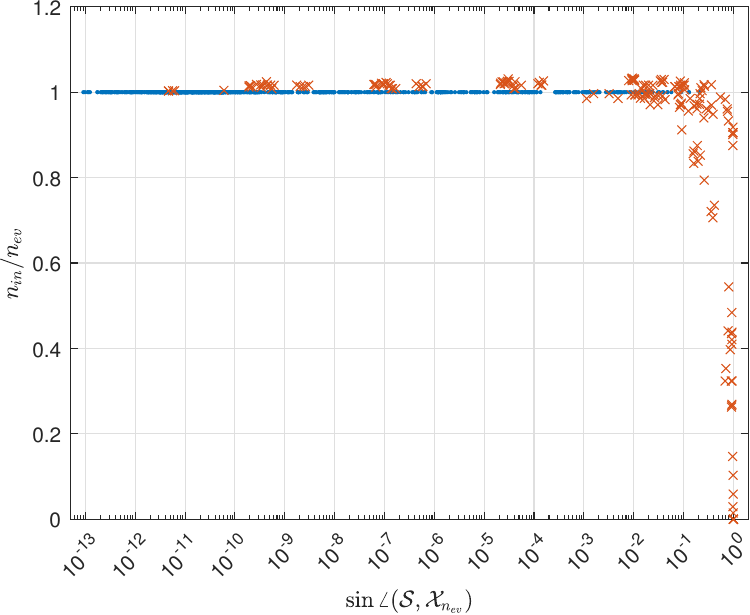}
    \subcaption{Residual norm-based removal with the threshold $\epsilon=10^{-2}$.}
    \label{subfig: nin/nev of different approaches: residual 1e-2}
  \end{subfigure}
  \begin{subfigure}[t]{\txtW}
    \includegraphics[width=\textwidth]{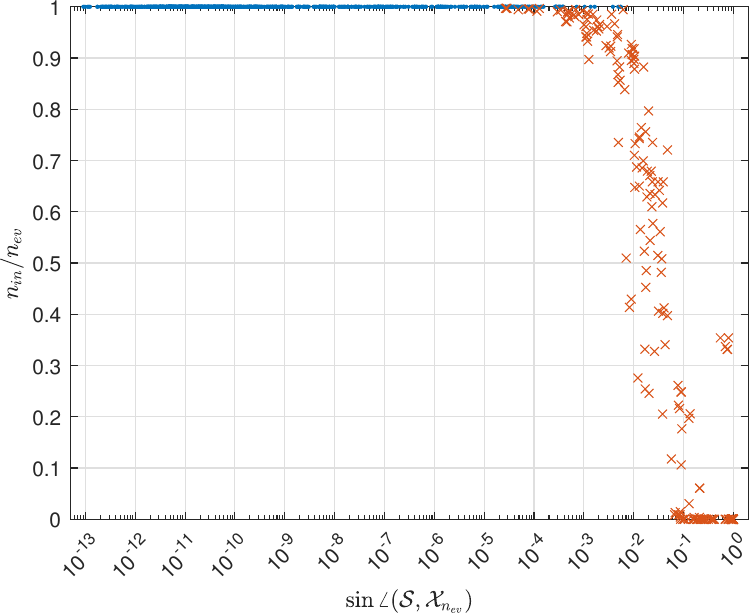}
    \subcaption{Residual norm-based removal with the threshold $\epsilon=10^{-4}$.}
    \label{subfig: nin/nev of different approaches: residual 1e-4}
  \end{subfigure}

  \caption{
    The performance of different removal approaches.
    The horizontal axis is the deviation $\epsilon_{ev}$ in \eqref{deviation},
    and the vertical axis is the ratio $n_{in} / n_{ev}$.
    In each figure, the pairs $(\epsilon_{ev}, n_{in}/n_{ev})$'s of all the
    test matrices and intervals in \cref{tab: test matrices} are plotted together.
    A pair $(\epsilon_{ev}, n_{in}/n_{ev})$ is marked by a
    dot if $n_{in} = n_{ev}$ and by a cross if $n_{in} \neq n_{ev}$.
  }
  \label{fig: nin/nev of different approaches}
\end{figure}

\cref{subfig: nin/nev of different approaches: TSVD 1e-12,subfig: nin/nev of different approaches: TSVD 1e-14}
shows that the CJ--SS--RR algorithm computed a number of spurious Ritz values
even if $\epsilon_{ev}$ dropped below $10^{-9}$.
Therefore, the appearance of spurious Ritz values are not unusual and, in fact,
are frequent events.

\cref{subfig: nin/nev of different approaches: residual 1e-2,subfig: nin/nev of different approaches: residual 1e-4}
show the results of the residual norm-based removal approach.
We can see that this approach with $\delta=10^{-4}$ succeeded to retain
$n_{ev}$ Ritz values in the
interval when the deviation $\epsilon_{ev}$ is small enough,
and $n_{in} = n_{ev}$ always when $\epsilon_{ev} < 10^{-5}$.
Compared with $\delta=10^{-2}$,
it is observed that when $\delta=10^{-4}$,
the number $n_{in}$ of the retained Ritz values is never larger than $n_{ev}$.
This indicates that the residual norm-based removal approach with $\delta=10^{-4}$ does not miss any spurious Ritz values.
For $\delta=10^{-2}$, however,
even when the deviation $\epsilon_{ev}$ is as small as $10^{-11}$, we still had $n_{in}>n_{ev}$,
indicating that the residual norm-based removal approach with $\delta=10^{-2}$ failed
to remove some spurious Ritz values. On the other hand,
for $\delta=10^{-4}$, one could have $n_{in} \ll n_{ev}$ if $\epsilon_{ev}$ is not small.
Thus, the residual norm-based removal approach is sensitive to the threshold $\delta$
and works poorly, and
different $\delta$ may lead to completely different computational results.

In summary, the theoretical analysis and numerical results have demonstrated
that the TSVD approach often does not work well.
For the residual norm-based removal approach,
it is very hard to choose a proper threshold $\delta$ that works well for all problems.
In contrast,
the theory and experiments have validated the effectiveness of
the refined removal approach for the restarted CJ--SS--RRR algorithm. Besides,
the refined removal approach is tune free, and does not need any auxiliary user-prescribed parameters,
making it more robust and user-friendly than the other two removal approaches used in the restarted SS--RR
algorithms.

\subsection{The comparison of CJ--SS--RR and CJ--SS--RRR}

In this subsection,
we compare the performance of the restarted CJ--SS--RR and CJ--SS--RRR algorithms.
% The refined method aims to improve the selection of Ritz vectors by addressing the limitations of the original method.
We test the two algorithms for the moment number $M = 4, 8$
and the dimension $M\ell = 8 \lceil 1.5 n_{ev} / 8 \rceil$ of the search subspace
such that $M\ell \geq 1.5 n_{ev}$.
The series degree $d$ is determined by \cref{eq: degree estimate for practice} with $D = 2$ and $K = 5$.
Additionally,
the tolerance $tol$ for the convergence test in \cref{alg: CJSSRR,alg: CJSSRRR} is set to $10^{-12}$,
and the threshold $\delta$ for the residual norm-based removal approach in \cref{alg: CJSSRR} is set to $10^{-4}$ as in \cite{CJ_SS_RR}.
For each test matrix and interval in \cref{tab: test matrices},
we generated ten random starting matrices $V^{(0)}$ and ran the two algorithms until convergence.

\begin{table}[htbp]
\caption{
  The average restarts used for  restarted CJ--SS--RR and CJ--SS--RRR with $tol = 10^{-12}$.
  The threshold $\delta$ in CJ--SS--RR is $10^{-4}$.
  Both algorithms have converged for all the test problems and obtained $n_{ev}$ converged
  eigenvalues in the interval of interest with the relative residual norms smaller than $tol$.
}
\label{tab: average restartsof CJSSRR and refined CJSSRR}
\centering
\small
\begin{tabular}{clccccc}
\toprule
Interval            & Matrix        & \multicolumn{2}{c}{$M = 4$} &  & \multicolumn{2}{c}{$M = 8$} \\
\cline{3-4}\cline{6-7}
                    &               & CJ--SS--RR     & CJ--SS--RRR   &  & CJ--SS--RR     & CJ--SS--RRR           \\
\midrule
\multirow{7}{*}{I}  & SiH4          & 3.1            & 3.1           &  & 3.1                & 3                 \\
                    & SiNa          & 4              & 4             &  & 4                  & 3.4               \\
                    & delaunay\_n13 & 4              & 4             &  & 4                  & 3.4               \\
                    & benzene       & 4              & 3.6           &  & 3.8                & 3                 \\
                    & stokes64      & 3.2            & 3             &  & 3                  & 3                 \\
                    & Pres\_Poisson & 6.7            & 6.5           &  & 5                  & 5                 \\
                    & Si10H16       & 4              & 3             &  & 3.5                & 3                 \\
\midrule
\multirow{7}{*}{II} & SiH4          & 4              & 3             &  & 3.4                & 3                 \\
                    & SiNa          & 4.4            & 4             &  & 4.1                & 3.2               \\
                    & delaunay\_n13 & 5.7            & 5             &  & 5.1                & 4                 \\
                    & benzene       & 4.6            & 4             &  & 4.2                & 3                 \\
                    & stokes64      & 5.4            & 5             &  & 4.5                & 4                 \\
                    & Pres\_Poisson & 4.3            & 4             &  & 3.9                & 3.1               \\
                    & Si10H16       & 4.3            & 3.4           &  & 4.1                & 3                 \\
\bottomrule
\end{tabular}

\end{table}

\cref{tab: average restartsof CJSSRR and refined CJSSRR} shows the average restarts
used for the convergence of the two restarted algorithms.
From \cref{tab: average restartsof CJSSRR and refined CJSSRR},
we can see that for all the test problems, CJ--SS--RRR
did not require more restarts to converge than CJ--SS--RR,
and in most cases, it consumed fewer restarts to converge.
This is due to the fact that the residual norms of the refined Ritz pairs are smaller than those of the Ritz pairs, making CJ--SS--RRR converge more quickly.
Moreover, since the cost of constructing the moment matrix $S$ in \cref{eq: def: polynomial approximation of S} overwhelms the other computational,
CJ--SS--RRR not only improved the accurate determination of the $n_{ev}$ desired eigenvalues
but also is more efficient.

\section{Conclusion}\label{sec:7}

Based on the refined Rayleigh--Ritz projection, we have proposed the SS--RRR methods,
which includes the CJ--SS--RRR method as a special instance. We have presented
an effective and robust removal approach to identifying
spurious Ritz values and refined Ritz vectors and to reliably
retaining all the genuine (refined) Ritz values and removing
extra refined Ritz vectors. Such removal approach is much more effective and robust
than the TSVD removal approach and
the residual norm-based removal approach in restarted SS-RR algorithms, and it is
tune free and easily implemented.

We have conducted detailed
numerical experiments on several test problems to demonstrate the performance of
CJ--SS--RRR, showing that the refined removal approach
removed the extra refined Ritz vectors reliably as (refined) Ritz values started to
converge, while the TSVD approach and residual norm-based removal approaches either failed or were
susceptible to truncation tolerances or user-prescribed thresholds $\delta$.
Furthermore, CJ--SS--RRR outperforms CJ--SS--RR
and is a more effective and efficient eigensolver
for computing partial eigenpairs of large-scale Hermitian matrices than CJ--SS--RR.

%\bibliographystyle{siamplain}
%\bibliography{references}

\begin{thebibliography}{10}

\bibitem{bai2000}
{\sc Z.~Bai, J.~Demmel, J.~Dongarra, A.~Ruhe, and H.~A. Van~der Vorst}, {\em
  Templates for the Solution of Algebraic Eigenvalue Problems: A Practical
  Guide}, SIAM, Philadelphia, PA, 2000,
  \url{https://doi.org/10.1137/1.9780898719581}.

\bibitem{matrix_collection}
{\sc T.~A. Davis and Y.~Hu}, {\em The {U}niversity of {Florida} sparse matrix
  collection}, ACM Trans. Math. Software, 38 (2011),
  \url{https://doi.org/10.1145/2049662.2049663}.

\bibitem{eig_count_Saad}
{\sc E.~Di~Napoli, E.~Polizzi, and Y.~Saad}, {\em Efficient estimation of
  eigenvalue counts in an interval}, Numer. Linear Algebra Appl., 23 (2016),
  pp.~674--692, \url{https://doi.org/10.1002/nla.2048}.

\bibitem{SS_package}
{\sc Y.~Futamura and T.~Sakurai}, {\em {z-Pares}: Parallel eigenvalue solver},
  2014, \url{https://zpares.cs.tsukuba.ac.jp/}.

\bibitem{Golub_Matrix}
{\sc G.~H. Golub and C.~F. Van~Loan}, {\em Matrix Computations}, The John
  Hopkins University Press, Baltimore, 4th~ed., 2013,
  \url{https://doi.org/10.1137/1.9781421407944}.

\bibitem{FEAST_Zolotarev}
{\sc S.~G\"{u}ttel, E.~Polizzi, P.~T.~P. Tang, and G.~Viaud}, {\em {Zolotarev}
  quadrature rules and load balancing for the {FEAST} eigensolver}, SIAM J.
  Sci. Comput., 37 (2015), pp.~A2100--A2122,
  \url{https://doi.org/10.1137/140980090}.

\bibitem{block_SS_RR}
{\sc T.~Ikegami and T.~Sakurai}, {\em Contour integral eigensolver for
  {non-Hermitian} systems: a {Rayleigh--Ritz-type} approach}, Taiwanese J.
  Math., 14 (2010), pp.~825--837,
  \url{https://doi.org/10.11650/twjm/1500405869}.

\bibitem{SS_theory2010}
{\sc T.~Ikegami, T.~Sakurai, and U.~Nagashima}, {\em A filter diagonalization
  for generalized eigenvalue problems based on the {Sakurai--Sugiura}
  projection method}, J. Comput. Appl. Math., 233 (2010), pp.~1927--1936,
  \url{https://doi.org/10.1016/j.cam.2009.09.029}.

\bibitem{SS_Arnoldi}
{\sc A.~Imakura, L.~Du, and T.~Sakurai}, {\em A block {Arnoldi-type} contour
  integral spectral projection method for solving generalized eigenvalue
  problems}, Appl. Math. Lett., 32 (2014), pp.~22--27,
  \url{https://doi.org/10.1016/j.aml.2014.02.007}.

\bibitem{SS_theory2016}
{\sc A.~Imakura, L.~Du, and T.~Sakurai}, {\em Error bounds of {Rayleigh--Ritz}
  type contour integral-based eigensolver for solving generalized eigenvalue
  problems}, Numer. Algorithms, 71 (2016), pp.~103--120,
  \url{https://doi.org/10.1007/s11075-015-9987-4}.

\bibitem{SS_review}
{\sc A.~Imakura, L.~Du, and T.~Sakurai}, {\em Relationships among contour
  integral-based methods for solving generalized eigenvalue problems}, Jpn. J.
  Ind. Appl. Math., 33 (2016), pp.~721--750,
  \url{https://doi.org/10.1007/s13160-016-0224-x}.

\bibitem{SS_2023}
{\sc A.~Imakura and T.~Sakurai}, {\em Complex moment-based eigensolver coupled
  with two {Krylov} subspaces}, J. Comput. Appl. Math., 432 (2023), 115283,
  \url{https://doi.org/10.1016/j.cam.2023.115283}.

\bibitem{Jackson_damp}
{\sc L.~O. Jay, H.~Kim, Y.~Saad, and J.~R. Chelikowsky}, {\em Electronic
  structure calculations for plane-wave codes without diagonalization}, Comput.
  Phys. Commun., 118 (1999), pp.~21--30,
  \url{https://doi.org/10.1016/S0010-4655(98)00192-1}.

\bibitem{jia1997refined}
{\sc Z.~Jia}, {\em Refined iterative algorithms based on {Arnoldi}'s process
  for large unsymmetric eigenproblems}, Linear Algebra Appl., 259 (1997),
  pp.~1--23, \url{https://doi.org/10.1016/S0024-3795(96)00238-8}.

\bibitem{jia1998refined}
{\sc Z.~Jia}, {\em A refined iterative algorithm based on the block {Arnoldi}
  process for large unsymmetric eigenproblems}, Linear Algebra Appl., 270
  (1998), pp.~171--189, \url{https://doi.org/10.1016/S0024-3795(97)00023-2}.

\bibitem{jia_multiple}
{\sc Z.~Jia}, {\em {Arnoldi} type algorithms for large unsymmetric multiple
  eigenvalue problems}, J. Comput. Math., 17 (1999), pp.~257--274,
  \url{https://doi.org/10.1006/jagm.1998.0996}.

\bibitem{jia1999refined}
{\sc Z.~Jia}, {\em Polynomial characterizations of the approximate eigenvectors
  by the refined {Arnoldi} method and an implicitly restarted refined {Arnoldi}
  algorithm}, Linear Algebra Appl., 287 (1999), pp.~191--214,
  \url{https://doi.org/10.1016/S0024-3795(98)10197-0}.

\bibitem{jia2000refined}
{\sc Z.~Jia}, {\em A refined subspace iteration algorithm for large sparse
  eigenproblems}, Appl. Numer. Math., 32 (2000), pp.~35--52,
  \url{https://doi.org/10.1016/S0168-9274(99)00008-2}.

\bibitem{refined_RR_theory}
{\sc Z.~Jia}, {\em Some theoretical comparisons of refined {Ritz} vectors and
  {Ritz} vectors}, Sci. China Ser. A, 47 (2004), pp.~222--233,
  \url{https://doi.org/10.1360/04za0020}.

\bibitem{CJ_SS_RR}
{\sc Z.~Jia and T.~Liu}, {\em A {Chebyshev--Jackson} series based block
  {SS--RR} algorithm for computing partial eigenpairs of real symmetric
  matrices}, arXiv preprint arXiv:2508.20456,  (2025).

\bibitem{jiastewart2001}
{\sc Z.~Jia and G.~W. Stewart}, {\em An analysis of the {Rayleigh--Ritz} method
  for approximating eigenspaces}, Math. Comput., 70 (2001), pp.~637--647,
  \url{https://doi.org/10.1090/S0025-5718-00-01208-4}.

\bibitem{CJ_FEAST_cross}
{\sc Z.~Jia and K.~Zhang}, {\em A {FEAST} {SVDsolver} based on
  {Chebyshev--Jackson} series for computing partial singular triplets of large
  matrices}, J. Sci. Comput., 97 (2023), 21,
  \url{https://doi.org/10.1007/s10915-023-02342-y}.

\bibitem{CJ_FEAST_augmented}
{\sc Z.~Jia and K.~Zhang}, {\em An augmented matrix-based {CJ-FEAST}
  {SVDsolver} for computing a partial singular value decomposition with the
  singular values in a given interval}, SIAM J. Matrix Anal. Appl., 45 (2024),
  pp.~24--58, \url{https://doi.org/10.1137/23M1547500}.

\bibitem{CJ_FEAST_GSVD}
{\sc Z.~Jia and K.~Zhang}, {\em A {CJ-FEAST GSVD}solver for computing a partial
  gsvd of a large matrix pair with the generalized singular values in a given
  interval}, Numer. Math., 157 (2025), pp.~897--949,
  \url{https://doi.org/10.1007/s00211-025-01466-7}.

\bibitem{Jin_FEM_EM}
{\sc J.-M. Jin}, {\em The Finite Element Method in Electromagnetics}, Jhon
  Wiley \& Sons, Inc., Hoboken, NJ, 3rd~ed., 2014.

\bibitem{FEAST_non_Hermitian}
{\sc J.~Kestyn, E.~Polizzi, and P.~T. Peter~Tang}, {\em {FEAST} eigensolver for
  {non-Hermitian} problems}, SIAM J. Sci. Comput., 38 (2016), pp.~S772--S799,
  \url{https://doi.org/10.1137/15M1026572}.

\bibitem{Martin_Electronic_Structure}
{\sc R.~M. Martin}, {\em Electronic Structure: Basic Theory and Practical
  Methods}, Cambridge University Press, Cambridge, UK, 2nd~ed., 2020,
  \url{https://doi.org/10.1017/9781108555586}.

\bibitem{Parlett_Symmetric}
{\sc B.~N. Parlett}, {\em The Symmetric Eigenvalue Problem}, SIAM,
  Philadelphia, PA, 1998, \url{https://doi.org/10.1137/1.9781611971163}.

\bibitem{FEAST_SI}
{\sc P.~T. Peter~Tang and E.~Polizzi}, {\em {FEAST} as a subspace iteration
  eigensolver accelerated by approximate spectral projection}, SIAM J. Matrix
  Anal. Appl., 35 (2014), pp.~354--390,
  \url{https://doi.org/10.1137/13090866X}.

\bibitem{FEAST2009}
{\sc E.~Polizzi}, {\em Density-matrix-based algorithm for solving eigenvalue
  problems}, Phys. Rev. B, 79 (2009), 115112,
  \url{https://doi.org/10.1103/PhysRevB.79.115112}.

\bibitem{FEAST_package}
{\sc E.~Polizzi}, {\em {FEAST} eigenvalue solver v4.0 user guide}, 2020,
  \url{https://doi.org/10.48550/arXiv.2002.04807}.

\bibitem{ApproxIntroduction}
{\sc T.~J. Rivlin}, {\em An Introduction to the Approximation of Functions},
  Dover, New York, 1981.

\bibitem{Saad_Linear}
{\sc Y.~Saad}, {\em Iterative Methods for Sparse Linear Systems}, SIAM,
  Philadelphia, PA, 2nd~ed., 2003,
  \url{https://doi.org/10.1137/1.9780898718003}.

\bibitem{Saad_Eigen}
{\sc Y.~Saad}, {\em Numerical Methods for Large Eigenvalue Problems, Revised
  Edition}, SIAM, Philadelphia, PA, 2011,
  \url{https://doi.org/10.1137/1.9781611970739}.

\bibitem{SS_restart1}
{\sc T.~Sakurai, Y.~Futamura, and H.~Tadano}, {\em Efficient parameter
  estimation and implementation of a contour integral-based eigensolver}, J.
  Algorithms Comput. Technol., 7 (2013), pp.~249--269,
  \url{https://doi.org/10.1260/1748-3018.7.3.249}.

\bibitem{SS_Hankel}
{\sc T.~Sakurai and H.~Sugiura}, {\em A projection method for generalized
  eigenvalue problems using numerical integration}, J. Comput. Appl. Math., 159
  (2003), pp.~119--128, \url{https://doi.org/10.1016/S0377-0427(03)00565-X}.

\bibitem{SS_RR}
{\sc T.~Sakurai and H.~Tadano}, {\em {CIRR}: a {Rayleigh--Ritz} type method
  with contour integral for generalized eigenvalue problems}, Hokkaido Math.
  J., 36 (2007), pp.~745--757, \url{https://doi.org/10.14492/hokmj/1272848031}.

\bibitem{Stewart_Eigen}
{\sc G.~W. Stewart}, {\em Matrix Algorithms, Volume II: Eigensystems}, SIAM,
  Philadelphia, PA, 2001, \url{https://doi.org/10.1137/1.9780898718058}.

\bibitem{szyld2006}
{\sc D.~B. Szyld}, {\em The many proofs of an identity on the norm of oblique
  projections}, Numer. Algor., 42 (2006), pp.~309--323,
  \url{https://doi.org/10.1007/s11075-006-9046-2}.

\bibitem{Trefethen_Embree}
{\sc L.~N. Trefethen and M.~Embree}, {\em Spectra and Pseudospectra: The
  Behavior of Nonnormal Matrices and Operators}, Princeton University Press,
  Princeton, 2005, \url{https://doi.org/10.1515/9780691213101}.

\bibitem{book:VanDerVorst2002}
{\sc H.~A. van~der Vorst}, {\em {Computational Methods for Large Eigenvalue
  Problems}}, Handbook of Numerical Analysis, Vol. VIII, edited by P. G.
  Ciarlet and J. L. Lions, Elsvier, 2002,
  \url{https://doi.org/10.1016/S1570-8659(02)08003-1}.

\bibitem{FEAST_oblique}
{\sc G.~Yin, R.~H. Chan, and M.-C. Yeung}, {\em A {FEAST} algorithm with
  oblique projection for generalized eigenvalue problems}, Numer. Linear
  Algebra Appl., 24 (2017), e2092, \url{https://doi.org/10.1002/nla.2092}.

\end{thebibliography}

\end{document}